\documentclass[12pt]{article}

\usepackage{amsmath,amsthm,amsfonts,amssymb,latexsym,amscd,bbm}
\usepackage{color}
\usepackage{array}

\newcolumntype{M}{>{$}c<{$}}

\input{prepictex.tex}
\input{pictex.tex}
\input{postpictex.tex}

\begin{document}

\newtheorem{theorem}{Theorem}[section]
\newtheorem{corollary}[theorem]{Corollary}
\newtheorem{lemma}[theorem]{Lemma}
\newtheorem{proposition}[theorem]{Proposition}
\newtheorem{conjecture}[theorem]{Conjecture}
\newtheorem{commento}[theorem]{Comment}
\newtheorem{definition}[theorem]{Definition}
\newtheorem{problem}[theorem]{Problem}
\newtheorem{remark}[theorem]{Remark}
\newtheorem{remarks}[theorem]{Remarks}
\newtheorem{example}[theorem]{Example}
\newtheorem{assumption}[theorem]{Standing Assumption}
\newtheorem*{assumption*}{Standing Assumption}

\newcommand{\Nb}{{\mathbb{N}}}
\newcommand{\Rb}{{\mathbb{R}}}
\newcommand{\Tb}{{\mathbb{T}}}
\newcommand{\Zb}{{\mathbb{Z}}}
\newcommand{\Cb}{{\mathbb{C}}}

\newcommand{\Af}{\mathfrak A}
\newcommand{\Bf}{\mathfrak B}
\newcommand{\Ef}{\mathfrak E}
\newcommand{\Gf}{\mathfrak G}
\newcommand{\Hf}{\mathfrak H}
\newcommand{\Kf}{\mathfrak K}
\newcommand{\Lf}{\mathfrak L}
\newcommand{\Mf}{\mathfrak M}
\newcommand{\Rf}{\mathfrak R}

\newcommand{\x}{\mathfrak x}

\def\A{{\mathcal A}}
\def\B{{\mathcal B}}
\def\C{{\mathcal C}}
\def\D{{\mathcal D}}
\def\E{{\mathcal E}}
\def\F{{\mathcal F}}
\def\G{{\mathcal G}}
\def\H{{\mathcal H}}
\def\J{{\mathcal J}}
\def\K{{\mathcal K}}
\def\LL{{\mathcal L}}
\def\N{{\mathcal N}}
\def\M{{\mathcal M}}
\def\N{{\mathcal N}}
\def\OO{{\mathcal O}}
\def\P{{\mathcal P}}
\def\Q{{\mathcal Q}}
\def\SS{{\mathcal S}}
\def\T{{\mathcal T}}
\def\U{{\mathcal U}}
\def\W{{\mathcal W}}
\def\Z{{\mathcal Z}}

\def\ext{\operatorname{Ext}}
\def\spn{\operatorname{span}}
\def\clsp{\overline{\operatorname{span}}}
\def\Ad{\operatorname{Ad}}
\def\ad{\operatorname{Ad}}
\def\tr{\operatorname{tr}}
\def\id{\operatorname{id}}
\def\en{\operatorname{End}}
\def\aut{\operatorname{Aut}}
\def\out{\operatorname{Out}}
\def\coker{\operatorname{coker}}
\def\image{\operatorname{Im}}
\def\alg{{\mathcal Alg}}

\def\la{\langle}
\def\ra{\rangle}
\def\rh{\rightharpoonup}

\newcommand{\setof}[2]{\left\{ #1 \; \middle\vert \; #2 \right\}}
\newcommand{\vertexbox}[3]{\begin{array}{|c|c|} \hline \multicolumn{2}{|c|}{#1} \\ \hline #2 & #3 \\ \hline \end{array}}

\title{The Polynomial Endomorphisms of Graph Algebras}

\author{Rune Johansen\footnote{The first and the third author were supported by the Villum Fonden Research Network 
`Experimental mathematics in number theory, operator algebras and topology'.}, Adam P. W. S{\o}rensen and  
Wojciech Szyma{\'n}ski\footnote{This work was supported by  the DFF-Reesearch Project 2 on `Automorphisms and invariants of 
operator algebras', Nr. 7014--00145B.}}

\date{{\small \today}}

\maketitle

\renewcommand{\sectionmark}[1]{}

\vspace{7mm}
\begin{abstract}
We investigate polynomial endomorphisms of graph $C^*$-algebras and Leavitt path algebras. To this end, we define and analyze the coding graph 
corresponding to each such an endomorphism. We find an if and only if condition for the endomorphism to restrict to an automorphism of the diagonal MASA, which is stated in terms of synchronization of a certain labelling on the coding graph. We show that the dynamics induced this way on the space of infinite paths (the spectrum of the MASA) is generated by an asynchronous  transducer. 
\end{abstract}

\vfill\noindent {\bf MSC 2010}: 46L40, 16A72, 37B10

\vspace{3mm}

\newpage

\section{Introduction}

The aim of this paper is to study a class of endomorphisms of both graph $C^*$-algebras and  Leavitt path algebras.
To describe the class and motivate why we study it, we recall some history.
In \cite{Cun2} Cuntz studied unital endomorphisms of the Cuntz algebras $\OO_n$ and showed that there is a one-to-one correspondence between unitaries in $\OO_n$ and unital endomorphisms of $\OO_n$. 
Trying to understand all endomorphisms of $\OO_n$ seems an impossible task.
So instead, Cuntz' result has been used in the following way: 
Consider some subclass of particularly nice unitaries of $\OO_n$ and then investigate the endomorphisms they induce.
Some of the crucial questions arising in this context are: 
Is the endomorphism an automorphism on the entire algebra? 
Does it restrict to an automorphism on the diagonal?
What is the corresponding classical dynamics on the Cantor set?
Since all unital endomorphisms of $\OO_n$ are injective, the automorphism questions simplify to asking if the endomorphism is onto the appropriate sets.

In \cite{CStams}, a study of endomorphisms coming from permutation unitaries in the AF-core is carried out and in \cite{CHSam} endomorphisms coming from so-called polynomial unitaries are consider. 
There are a variety of applications, see for instance \cite{bj}, \cite{i},  and \cite{Kaw2}.

Graph $C^*$-algebras are a wide reaching generalization of the Cuntz algebras. 
The generalization happened stepwise, but the first definitions covering all the graphs we will consider appeared in \cite{KPR}. 
See the excelent monograph \cite{RaeburnBook} for a thorough introduction to the theory.
Cuntz's idea of relating unitary endomorphisms and unitaries also works in the generality of graph $C^*$-algebras (\cite{CHSjfa}), and so again we can study the endomorphisms given by nice unitaries. 
In \cite{AJSz}, a study of the permutative endomorphisms of graph $C^*$-algebras is carried out. 

In this paper, we will look at polynomial endomorphisms of graph $C^*$-algebras.  
We begin by observing that the constructions in \cite{CHSjfa, Cun2} are of an entirely algebraic nature.
Therefore we carry out our investigation in the purely algebraic setting of of Leavitt path algebras, introduced in \cite{AAP, AMP}, as this costs very little but gives new results.       
Following the ideas of \cite{AJSz} we start with a polynomial unitary $u_\J \in L_\Zb(E)$ and its associated endomorphisms $\Lambda_\J$, and build from it the \emph{coding graph} $E_\J$. 
This a finite directed graph with multiple labellings. 
We have three main applications of the coding graph: 
we show how to use the coding graph to easily read of $\Lambda_\J(S_\alpha)$ for any path $\alpha$ (Lemma \ref{imageofpath}); 
we give an if and only if condition for $\Lambda_\J$ to restrict to an automorphism of the diagonal in terms of the coding graph (Theorem \ref{diagonalauto});
when $\Lambda_\J$ restricts to an automorphism it induces a self homeomorphism $\psi$ of the infinite path space $E^\infty$, we modify the coding graph to build an asynchronous transducer (see \cite{russian} and \cite{bb}) that describes the action of $\psi$ (Section \ref{transducers}).    

Our condition for restricting to an automorphism on the diagonal is fully parallel to the criterion found in \cite{AJSz} for permutative endomorphisms. 
Unfortunately we cannot present a criterion for invertibility of a polynomial endomorphism on the entire graph algebra. 
However, we believe that coding graphs provide the correct language in which such a criterion perhaps may be phrased. 
A belief backed up by the otherwise great explanative power of the coding graph. 

In the case of permutative automorphisms of the diagonal, the action they induce on the infinite path space is quite well understood and eventually commutes 
with the shift.  
The asynchronous transducer we build to implement the action of a polynomial automorphisms of the diagonal  cannot be described that way, since 
the action may not be eventually commuting with the shift.
Hence we get new dynamical phenomena.  
We envisage applications of this construction to investigations of endomorphisms and automorphisms of the Thompson groups, following up on the recent work in this direction in \cite{BHSz}. 

We end the paper with a section showing concrete examples of our constructions.  


\section{Preliminaries}

\subsection{Finite directed graphs and their algebras}

Let $E=(E^0,E^1,r,s)$ be a directed (multi)graph with {\em finite} sets of vertices $E^0$ and edges $E^1$, respectively. 
Let $r,s \colon E^1 \to E^0$ be the range and source maps, respectively. 
A {\em path} $\mu$ of length $|\mu|=k$ is a sequence $\mu=\mu_1\mu_2\ldots\mu_k$ of $k$ edges $\mu_j$ such that $r(\mu_j)=s(\mu_{j+1})$ for $j=1,\ldots,k-1$. 
For two paths $\mu, \nu$ we write $\mu \prec \nu$ to denote that $\mu$ is an initial segment of $\nu$.
We view the vertices as paths of length $0$.
The set of all paths of length $k$ is denoted $E^k$, and $E^*$ is the collection of all finite paths. 
The range and the source maps naturally extend to  maps $r,s \colon E^*\to E^0$. 
A vertex is called a {\em sink} if it  emits no edges and a {\em source} if it receives no edges. 
By a {\em cycle} we mean a path $\mu$ of non-zero length such that $r(\mu)=s(\mu)$.
A cycle $\mu=\mu_1\ldots\mu_k$ {\em has an exit} if there is a $j$ such that $s(\mu_j)$ emits at least two distinct edges. 
Throughout this paper we make the following standing assumption. 

\begin{assumption*}\rm 
All graphs $E$ we consider are finite, without sinks or sources, and such that every cycle has an exit. 
\end{assumption*}

\noindent To each such a graph we can associate a $C^*$-algebra.

\begin{definition}[See \cite{KPR}] \label{def:gca}
The $C^*$-algebra $C^*(E)$ corresponding to a graph $E$  is the universal $C^*$-algebra 
generated by mutually orthogonal projections $p_v$, $v\in E^0$, and partial isometries $s_e$, $e\in E^1$, 
subject to the following relations:
\begin{description}
\item{(GCA1)} $s_e^*s_e = p_{r(e)}$ for all $e\in E^1$, and, 
\item{(GCA2)} $p_v = \displaystyle{\sum_{s(e)=v}} s_e s_e^*$ for all $v\in E^0$. 
\end{description} 
\end{definition}

\noindent We warn readers that there are two competing conventions for the definition of $C^*(E)$. 
We use the opposite convention to that of Raeburn's monograph \cite{RaeburnBook}.

For a path $\mu=\mu_1\ldots\mu_k$, we denote by $s_\mu = s_{\mu_1} \cdots s_{\mu_k}$ the corresponding partial isometry in $C^*(E)$.
Each $s_\mu$ is non-zero and has the domain projection $p_{r(\mu)}$. 
For convenience, we agree to write $s_v = p_v$ for a $v \in E^0$, and $s_\emptyset=1$.
The $C^*$-algebra $C^*(E)$ equals the closed span of $\{s_\mu s_\nu^* \mid \mu,\nu \in E^*\}$. 
Note that $s_\mu s_\nu^*\neq 0$ if and only if $r(\mu)=r(\nu)$. 

The range projections $p_\mu = s_\mu s_\mu^*$ of all partial isometries $s_\mu$ mutually commute and generate 
the {\em diagonal subalgebra} $\D_{C^*(E)}$ of $C^*(E)$. We set 
\[
  \D_{C^*(E)}^k = \spn\{p_\mu \mid |\mu|= k\}, \quad \text{ for } k \in \Nb.
\]
Then $\D_{C^*(E)}$ coincides with the norm closure of $\bigcup_{k=0}^\infty\D_{C^*(E)}^k$. 
If $E$ does not contain sinks and all cycles have exits (and this is what we always assume in the present paper) then $\D_{C^*(E)}$ is a MASA (maximal abelian subalgebra) in $C^*(E)$, \cite[Theorem 5.2]{HPP}. 

There exists a strongly continuous action $\gamma$ of the circle group $U(1)$ on $C^*(E)$, called the {\em gauge action}, such that $\gamma_z(s_e)=zs_e$ and $\gamma_z(p_v)=p_v$ for all $e\in E^1$, $v\in E^0$ and $z\in U(1)\subseteq\Cb$. 
The fixed-point algebra $C^*(E)^\gamma$ for the gauge action is an $AF$-algebra, denoted $\F_{C^*(E)}$ and called the {\em core} $AF$-subalgebra of $C^*(E)$.
It is the closed span of $\{S_\mu S_\nu^* \mid \mu,\nu\in E^*,\,|\mu|=|\nu|\}$. 
For $k\in\Nb$, we denote by $\F_{C^*(E)}^k$ the linear span of $\{S_\mu S_\nu^* \mid \mu,\nu\in E^*,\,|\mu|=|\nu|=k\}$. 
Then $\F_{C^*(E)}$ coincides with the norm closure of $\bigcup_{k=0}^\infty \F_{C^*(E)}^k$. 

For any unital commutative ring $R$, we can also associate an $R$-algebra to the graph $E$.

\begin{definition}[{\cite[Definitions 2.4 and 3.4]{T}}] \label{def:lpa}
Let $R$ be a commutative ring with unit and let $E$ be a graph. 
The Leavitt path algebra $L_R(E)$ of $E$ over $R$ is the universal $R$-algebra generated by pairwise orthogonal idempotents $\{ P_v \mid v \in E^0\}$ and elements $\{S_e, S_e^* \mid e \in E^1\}$ satisfying 
\begin{description}
 \item{(LPA1)} $S_e^* S_f = 0$, if $e \neq f$,
 \item{(LPA2)} $S_e^* S_e = r(e)$, 
 \item{(LPA3)} $P_{s(e)} S_e = S_e = S_e P_{r(e)}$, 
 \item{(LPA4)} $S_e^* P_{s(e)} = S_e^* = P_{r(e)} S_e^*$, and, 
 \item{(LPA5)} $P_v = \sum_{e \in s^{-1}(v)} S_e S_e^*$, if $s^{-1}(v)$ is finite and nonempty. \label{item:summation_relation}
\end{description}
\end{definition}

We will view $L_R(E)$ as a $*$-algebra, where the involution extends the map $S_e \mapsto S_e^*$.
If $R$ is a subring of $\Cb$ that is closed under conjugation, we will take the involution to be conjugate linear, in all other cases we take it to be linear. 

It is customary in the Leavitt path algebra literature to simply denote the element $S_e$ by $e$. 
This blurs the distinction between edges in $E$ and elements in $L_R(E)$, but has the advantage that important details are not hiding in subscripts.  
Since we will need to differentiate between the paths in the graph, the elements of the Leavitt path algebra, and the elements of the graph $C^*$-algebra we use the $S_e$ notation. 
To help distinguish between the Leavitt path algebra and the graph $C^*$-algebras we will use upper case letters for Leavitt path algebra related objects and lower case letters for the graph $C^*$-algebra related objects.
But only when this does not conflict with well established notation. 
For instance we still use $\D_{C^*(E)}$ for the diagonal of $C^*(E)$.

We follow the same labeling conventions for Leavitt path algebras as we did for graph $C^*$-algebras. 
So for a path $\mu=\mu_1\ldots\mu_k$, we write $S_\mu$ for the product $S_{\mu_1} \cdots S_{\mu_k}$.
Each $S_\mu$ is non-zero and has the domain projection $P_{r(\mu)}$. 
For convenience, we agree to write $S_v = P_v$ for a $v \in E^0$, and $S_\emptyset=1$.
Further, we recall that $L_R(E) = \spn_R\{s_\mu s_\nu^* \mid \mu,\nu \in E^*\}$. 
Note that $S_\mu S_\nu^* \neq 0$ if and only if $r(\mu)=r(\nu)$. 
All these facts are in \cite{T}.

For each $\mu \in E^*$ we let $P_\mu = S_\mu S_\mu^*$, which we call the range projection of $\mu$ following $C^*$-algebra conventions. 
It is indeed a projection, i.e. $P_\mu  = P_\mu^* = P_\mu^2$.
The range projections of all $S_\mu$ mutually commute and generate the {\em diagonal subalgebra} $\D_{L_R(E)}$.
We set 
\[
  \D_{L_R(E)}^k = \spn\{P_\mu \mid |\mu|= k\}, \quad \text{ for } k \in \Nb.
\]
Then $\D_{L_R(E)} = \bigcup_{k=0}^\infty\D_{L_R(E)}^k$. 
If $E$ does not contain sinks and all cycles have exits (and this is what we always assume in the present paper) then $\D_{L_R(E)}$ is a maximal abelian 
subalgebra (MASA) in $L_R(E)$ whenever $R$ is an integral domain \cite[Lemma 2.1]{ABHS}.

\begin{remark}\rm 
In the above, we have taken the \emph{generators and relations} point of view on graph $C^*$-algebras and Leavitt path algebras. 
Originally graph $C^*$-algebras were defined in terms of groupoids, \cite{KPRR}, and recently a groupoid picture for Leavitt path algebras 
has also emerged, \cite{CFST, Steinberg}.
It has often proved the case that some questions are answered more easily with generators and relations and some with groupoids. 
We shall stay clear of graph groupoids in the present paper, but mention them for completeness.
\end{remark}

We recall that all Leavitt path algebras are $\Zb$-graded, with the vertex projections being homogeneous of degree $0$ and elements of the from $S_\mu S_\nu^*$ being homogeneous of degree $|\mu| - |\nu|$.
We denote the set of $0$-graded elements by $\F_{L_R(E)}$.
Then $\F_{L_R(E)}$ is the span of $\{S_\mu S_\nu^* \mid \mu,\nu\in E^*,\,|\mu|=|\nu|\}$. 
For $k\in\Nb$, we denote by $\F_{L_R(E)}^k$ the linear span of $\{S_\mu S_\nu^* \mid \mu,\nu\in E^*,\,|\mu|=|\nu|=k\}$. 
Then $\F_E$ coincides with $\bigcup_{k=0}^\infty\F_{L_R(E)}^k$. 

The behavior of a Leavitt path algebra does to some degree depend on the ring of coefficients. 
This dependence is more subtle than it may appear at first, with many foundational results only assuming that the coefficients form a field. 
While we will not limit ourselves to fields, we will instead insist on integral domains with characteristic $0$, and our main interests are $\Zb$ and $\Cb$.
We note that as a consequence of the graded uinqueness theorem $L_\Zb(E)$ embeds naturally into $C^*(E)$ (see \cite[Lemma 3.4]{JS} and also \cite[Theorem 7.3]{T2}) and $L_R(E)$ whenever $R$ has characteristic $0$ (see \cite[Proposition 3.4 and Theorem 5.3]{T}).
The embedding simply maps generators to generators and we will often suppress it.
We have two main uses of these embeddings: 
We can work in $L_\Zb(E)$ when asking questions about $C^*(E)$ that are just about sums and products of generators. 
This gives us the advantage of more concrete elements. 
Alternatively, we can move problems from $L_\Zb(E)$ into $C^*(E)$ where there is more stucture to work with, such as positivity. 
To use these embeddings we impose the following standing assumption. 

\begin{assumption*}\rm 
All rings we consider will be assumed to be integral domains of characteristic $0$.
\end{assumption*}

In fact, in almost all cases the ring will be $\Zb$.

\begin{definition}
Let $E$ be a graph. 
A finite collection of paths $\{\mu_i\}_{i=1}^n \subseteq E^*$ is called a {\em partition} of a vertex $v$ if 
\[
    \sum_{i=1}^n S_{\mu_i} S_{\mu_i}^* = P_v, \text{ in } L_\Zb(E). 
\]
For a path $\nu\in E^*$, a finite collection of paths $\{\nu \mu_i \}_{i=1}^n \subseteq E^*$ is called a {\em partition} of $\nu$ if $\{\mu_i\}$ is a partition of $r(\nu)$.
In both cases, $\max\{|\mu_i|\}$ is said to be the {\em length} of the partition. 
\end{definition}

We note that a collection of paths $\{\mu_i\}_{i=1}^n \subseteq E^*$ form a partition if and only if 
\[
    \sum_{i=1}^n s_{\mu_i} s_{\mu_i}^* = p_v, \text{ in } C^*(E), 
\]
which happens if and only if 
\[
    \sum_{i=1}^n S_{\mu_i} S_{\mu_i}^* = P_v, \text{ in } L_R(E), 
\]
for some unital commutative ring $R$ of characteristic $0$.

\begin{remark}\rm
We defined the notion of a partition by using the Leavitt path algebra, but really it is entirely a graph concept. 
For each $\alpha \in E^*$ we define the \emph{cylinder set} of $\alpha$ as
\[
	\Z(\alpha) = \setof{\gamma \in E^\infty}{\alpha \prec \gamma},
\]
with the convention that if $\alpha$ is a path of length $0$, i.e. a vertex, then $\Z(\alpha)$ is the set of all path that start at $\alpha$.
One can show that a collection of paths $\{\mu_i\}_{i=1}^n \subseteq E^*$ is a partition of a vertex $v$ if and only if the sets $\Z(\mu_i)$ are pairwise disjoint and their union is $\Z(v)$.
\end{remark}


\subsection{Endomorphisms determined by unitaries}

Cuntz's classical approach to endomorphisms of $\OO_n$, \cite{Cun2}, has recently been extended to graph $C^*$-algebras in \cite{AJSz} and \cite{CHSjfa}. 
Many of these concepts can be transported almost verbatim to the setting of Leavitt path algebras, we do so below. 

For any unital $*$-algebra $A$ we let $\U(A)$ denote the collection of unitaries (elements $u$ for which $uu^* = 1 = u^*u$). 
We denote by $\U_{E}$ the collection of all those unitaries in $L_\Zb(E)$ which commute with all vertex projections $P_v$, $v \in E^0$. 
That is, 

\begin{equation}\label{ue}
\U_{E} = \U((\D_{L_\Zb(E)}^0)' \cap L_\Zb(E)). 
\end{equation}

If $u \in \U_E$ then $uS_e$, $e \in E^1$, their adjoints $S_e^* u^*$, and the projections $P_v$, $v\in E^0$, satisfy conditions (LPA1)-(LAP5).
Thus, by the universality of $L_\Zb(E)$, there exists a unital $*$-homomorphism $\Lambda_u \colon L_\Zb(E) \to L_\Zb(E)$ such that\footnote{The reader should be aware that in some papers (e.g. in \cite{Cun2}) a different convention is used, namely $\lambda_u(S_e)=u^*S_e$}

\begin{equation}\label{lambdau}
 \Lambda_u(S_e) = uS_e \;\; \text{and} \;\; \Lambda_u(P_v)=P_v, \;\; \text{for} \; e\in E^1, v\in E^0. 
\end{equation}

The mapping $u \mapsto \Lambda_u$ establishes a bijective correspondence between $\U_E$ and the semigroup of those unital $*$-endomorphisms of $L_\Zb(E)$ which fix all $P_v$, $v\in E^0$. 
Indeed, if $\Psi$ is a $*$-endomorphism that fixes all the $P_v$ and we put 
\[
 u = \sum_{e \in E^1} \Psi(S_e) S_e^*,
\]
then $\Psi = \Lambda_u$. 

Since all cycles in the graphs we consider have exits, the Cuntz-Krieger uniqueness theorem (\cite[Theorem 6.5]{T}) implies that all endomorphisms $\Lambda_u$ are automatically injective. 
We say that $\Lambda_u$ is {\em invertible} if $\Lambda_u$ is an automorphism of $L_\Zb(E)$.
That is, when it is also surjective. 

There is a natural extension of $\Lambda_u$ to an endomorphism of either $C^*(E)$, which we shall call $\lambda_u$, or $L_R(E)$, given simply by reading equation (\ref{lambdau}) as taking place in either $C^*(E)$ or $L_R(E)$.
I.e. by viewing $L_\Zb(E)$ as sitting inside $C^*(E)$ or $L_R(E)$.
In that viewpoint the following lemma simply records the fact that $L_\Zb(E)$ contains the generators of $C^*(E)$ and $L_R(E)$.

\begin{lemma}
Suppose we are given $u \in \U_E$. 
If $\Lambda_u$ is onto $L_\Zb(E)$, then $\lambda_u$ is onto $C^*(E)$ and when we view $\Lambda_u$ an an endomorphism of $L_R(E)$ it is also onto.  
\end{lemma}

We cannot prove that this implication can be reversed in general, i.e. it may well be possible that $\lambda_u$ is onto but $\Lambda_u$ is not. 
However, many previous proofs have established that $\lambda_u$ is onto in effect by showing that $\Lambda_u$ is onto. 
If we restrict our attention to the diagonal subalgebra the implication can be reversed.  

\begin{lemma} \label{masaisoiff}
Suppose we are given $u \in \U_E$. 
The restriction $\Lambda_u \vert_{\D_{L_\Zb(E)}}$ is onto $\D_{L_\Zb(E)}$ if and only if $\lambda_u \vert_{\D_{C^*(E)}}$ is onto $\D_{C^*(E)}$.   
\end{lemma} 
\begin{proof}
As above the forward implication follows from the fact the $\D_{L_\Zb(E)}$ generates $\D_{C^*(E)}$.
For the other implication, we note that all projections in $\D_{C^*(E)}$ live in $\D_{L_\Zb(E)}$ and that $\lambda_u$ maps projections to projections, so $\Lambda_u(\D_{L_\Zb(E)})$ contains all projections in $\D_{L_\Zb(E)}$. 
The result now follows from the fact that $\D_{L_\Zb(E)}$ is generated by its projections.
\end{proof}

We also consider an algebraic version of the {\em shift map}, introduced in \cite{CK}, given by 
\begin{equation}\label{shift}
  \Phi(x) = \sum_{e \in E^1} S_e x S_e^*, \;\;\; x \in L_\Zb(E). 
\end{equation}
The shift is a unital map. 
The graph $C^*$-algebra version is completely positive but we do not know of a natural algebraic version of that property. 
We note that the shift is an injective $*$-homomorphism when restricted to the relative commutant $(\D_{L_\Zb(E)}^0)'\cap L_\Zb(E)$ and that it globally preserves $\U_E$, $\F_{L_\Zb(E)}$ and $\D_{L_\Zb(E)}$. 

For $k\geq 1$ we denote

\begin{equation}\label{uk}
  u_k = u\Phi(u)\cdots\Phi^{k-1}(u). 
\end{equation}

For each $u \in \U_E$ and all $e \in E^1$ we have $S_e u=\Phi(u) S_e$, and thus 
\begin{equation} \label{lambdauS}
  \Lambda_u(S_\mu S_\nu^*)= u_{|\mu|} S_\mu S_\nu^* u_{|\nu|}^*
\end{equation}
for any two paths $\mu,\nu\in E^*$. 

In the present paper, we will be particularly concerned with a special class $\SS_E$ of unitaries in $\U_E$, see \cite{CHSjfa}, which we now define. 
Consider a finite subset  $\J \subseteq E^*\times E^*$ such that $s(\mu) = s(\nu)$ and $r(\mu) = r(\nu)$ for all $(\mu,\nu)\in\J$. 
We put
\begin{equation}\label{uj}
  u_\J = \sum_{(\mu,\nu)\in\J} S_\mu S_\nu^*. 
\end{equation} 
Note that the same element may admit many different presentations in the form $u_\J=u_{\J'}$ with $\J\neq\J'$. 
Choosing a convenient presentation will play an important role in what follows. 
Let $\J_1=\{\mu\in E^* \mid \exists\nu\in E^* \;\text{such that}\;(\mu,\nu)\in\J\}$ and similarly let $\J_2=\{\nu\in E^* \mid \exists\mu\in E^* \;\text{such that}\;(\mu,\nu)\in\J\}$. 
Note that $u_\J$ is unitary if and only if 

\begin{equation}\label{twopartitions}
  \sum_{\mu\in\J_1}P_\mu = 1 = \sum_{\nu\in\J_2}P_\nu. 
\end{equation}

Then $\SS_E$ is defined as the collection of all such unitary elements $u_\J$.
Assume $u_\J\in\SS_E$.
Then the endomorphisms $\Lambda_{u_\J}$ and $\lambda_{u_\J}$ will be denoted by $\Lambda_\J$ and $\lambda_\J$, respectively. 
Also, if $\mu \in \J_1$ then there is a unique $\nu \in J_2$ such that $(\mu,\nu)\in\J$, we write $J_\mu = (\mu,\nu)$. 
We will often write this element $(\mu,\nu)$ in the form $(\mu,e\kappa)$ with $e\in E^1$ and $\kappa\in E^*\cup\{\emptyset\}$ such that $\nu=e\kappa$.  

The following simple lemma follows immediately from the definition of a unitary $u_\J \in \SS_E$, so we omit its proof 
but record it for future reference. 

\begin{lemma}\label{partitions}
Let $u_\J\in\SS_E$. 
If $\nu$ is a prefix of $\mu \in \J_1$, then $\J_1$ contains a partition of $\nu$. 
If $\nu$ is a prefix of $\mu \in \J_2$, then $\J_2$ contains a partition of $\nu$. 
\end{lemma}

Since each unitary $u_\J \in \SS_E$ normalizes the diagonal MASA $\D_{C^*(E)}$ of $C^*(E)$, it follows that $\lambda_\J(\D_{C^*(E)}) \subseteq \D_{C^*(E)}$, \cite{CHSjfa}.
Hence we must also have that $\Lambda_J(\D_{L_\Zb(E)}) \subseteq \D_{L_\Zb(E)}$.
For $\Lambda_\J$ to be an automorphism of $L_{\Zb}(E)$ it is then {\em necessary} that $\Lambda_\J$ restricts to an automorphism of $\D_{L_\Zb(E)}$ (which happens precisely when $\Lambda_\J(\D_{L_\Zb(E)})=\D_{L_\Zb(E)}$. 
However, it is entirely possible that $\lambda_\J(\D_{C^*(E)})=\D_{C^*(E)}$ but $\lambda_\J\not\in\aut(C^*(E))$.  
Lots of concrete examples of this kind were given for the case of $\OO_2$ in \cite[Section 5]{CStams}, and for other Cuntz algebras in \cite{cksz}. Example \ref{exam3}, given at the end of this paper, is also of this kind. 


\section{Coding graphs}

\begin{definition}\label{codinggraph}
Let $u_\J\in\SS_E$. We define a labeled graph $(E_\J,\LL_J)$, called the {\em coding} graph of $u_\J$, with the 
vertex set and the edge set as follows:
$$ \begin{aligned}
E_\J^0 & = \J, \\
E_\J^1 & = \{ [(\mu_1,e_1\nu_1),(\mu_2,e_2\nu_2)] \mid S_{\nu_1}^*S_{\mu_2}\neq 0,\, 
(\mu_1,e_1\nu_1),(\mu_2,e_2\nu_2)\in\J\}. 
\end{aligned} $$
The source map, the range map, and the labeling $\LL_\J \colon E_\J^1\to L_\Zb(E)$ are defined as:
$$ \begin{aligned}
s_\J([(\mu_1,e_1\nu_1),(\mu_2,e_2\nu_2)]) & = (\mu_1,e_1\nu_1), \\
r_\J([(\mu_1,e_1\nu_1),(\mu_2,e_2\nu_2)]) & = (\mu_2,e_2\nu_2), \\
\LL_\J([(\mu_1,e_1\nu_1),(\mu_2,e_2\nu_2)]) & = S_{\nu_1}^*S_{\mu_2}. 
\end{aligned} $$
\end{definition}

We note that $S_{\nu_1}^*S_{\mu_2}\neq 0$ implies $s(\nu_1)=s(\mu_2)$. Since $u_\J\in\SS_E$, we also have 
$r(e_1)=s(\nu_1)$ and $s(\mu_2)=s(e_2)$, Consequently,  $r(e_1)=s(e_2)$. We also note that in fact 
$$ \LL_\J(E_\J^1)\subseteq\{S_\alpha \mid \alpha\in E^*\} \cup\{P_v \mid v\in E^0\} \cup \{S_\alpha^* \mid \alpha\in E^*\}. $$
In illustrations, a vertex $(\mu,e\nu)\in E_\J^0$ will be represented by the following box:
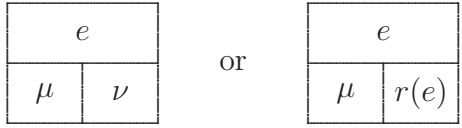
\[ \beginpicture
\setcoordinatesystem units <1cm,1cm>
\setplotarea x from -3 to 3, y from -0.9 to 0.9
\setlinear 
\plot 1 0.8  3 0.8  3 -0.8  1 -0.8  1 0.8 /
\plot 1 0  3 0 /
\plot 2 -0.8  2 0 /
\plot -3 0.8  -1 0.8  -1 -0.8  -3 -0.8  -3 0.8 /
\plot -3 0  -1 0 /
\plot -2 0  -2 -0.8 /
\put {$e$} at -2 0.4
\put {$e$} at 2 0.4 
\put {$\mu$} at 1.5 -0.4 
\put {$\mu$} at -2.5 -0.4
\put {$\nu$} at -1.5 -0.4
\put {or} at 0 0
\put {$r(e)$} at 2.5 -0.4
\endpicture \]
if $|\nu|\geq1$ or $\nu=\emptyset$, respectively. 

\begin{example}{\rm 
Consider the unitary element $u= S_1 S_{22}^* + S_{21} S_{21}^* + S_{22} S_1^*$ of the Leavitt algebra $L_{2,\Zb}$, given by the graph consisting of 
one vertex denoted $\emptyset$ and two edges denoted $1$ and $2$, respectively. 
If we put $\J = \{ (1,22), (21,21), (22,1) \}$ then $u = u_\J$
The coding graph of $u_\J$ looks as follows:
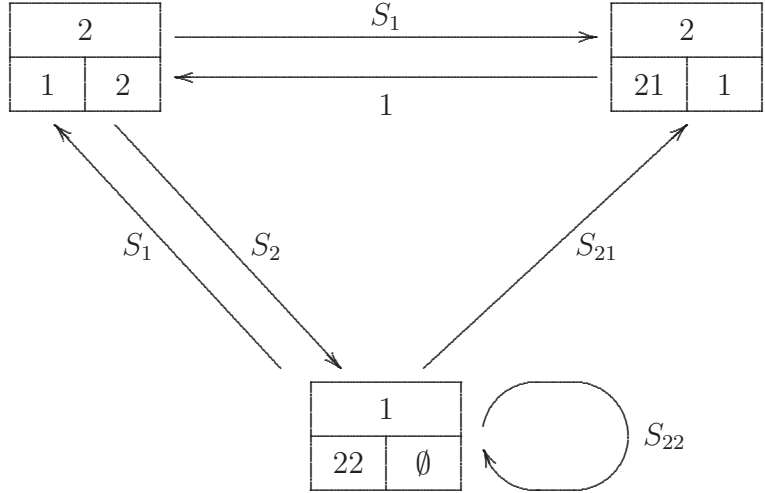
\[ \beginpicture
\setcoordinatesystem units <1cm,0.9cm>
\setplotarea x from -5 to 5, y from -3.7 to 3.7
\put {$2$} at -4 3.2
\put {$2$} at 4 3.2
\put {$1$} at 0 -2.4
\put {$1$} at -4.5 2.4
\put {$21$} at 3.5 2.4
\put {$22$} at -0.5 -3.2
\put {$2$} at -3.5 2.4
\put {$1$} at 4.5 2.4
\put {$S_{21}$} at 2.8 0
\put {$S_2$} at -1.6 0
\put {$S_1$} at -3.3 0
\put {$S_1$} at 0 3.4
\put {$1$} at 0 2.1
\put {$S_{22}$} at 3.7 -2.8
\put {$\emptyset$} at 0.5 -3.2
\arrow <0.235cm> [0.2,0.6] from 1.38 -3.2   to 1.3 -3
\arrow <0.235cm> [0.2,0.6] from 3.8 1.6  to 4 1.8
\arrow <0.235cm> [0.2,0.6] from -0.78 -1.6  to -0.6 -1.8 
\arrow <0.235cm> [0.2,0.6] from -4.22 1.6  to -4.4 1.8
\arrow <0.235cm> [0.2,0.6] from 2.6 3.1  to 2.8 3.1
\arrow <0.235cm> [0.2,0.6] from -2.6 2.5 to -2.8 2.5
\circulararc  180 degrees from 2.5 -3.6  center at 2.5 -2.8
\circulararc 80 degrees from 2 -2 center at 2 -2.8
\circulararc 80 degrees from 1.3 -3 center at 2 -2.8
\setlinear
\plot -1 -2  1 -2  1 -3.6  -1 -3.6  -1 -2 /
\plot -5 3.6  -3 3.6  -3 2  -5 2  -5 3.6 /
\plot 5 3.6  5 2  3 2  3 3.6  5 3.6 /
\plot -5 2.8  -3 2.8 /
\plot 3 2.8  5 2.8 /
\plot -1 -2.8  1 -2.8 /
\plot -4 2.8  -4 2 /
\plot 4 2.8  4 2 /
\plot 0 -3.6  0 -2.8 /
\plot 0.5 -1.8  4 1.8 /
\plot -0.6 -1.8  -3.6 1.8 /
\plot -1.4 -1.8  -4.4 1.8 /
\plot -2.8 2.5  2.8 2.5 /
\plot -2.8 3.1  2.8 3.1 /
\plot 2 -3.6  2.5 -3.6 /
\plot 2 -2  2.5 -2 /
\endpicture \]
}\end{example}

We will use the shorthand notation 
\begin{equation}\label{pathinej}
 (\mu_1,e_1\nu_1) \Rightarrow (\mu_k,e_k\nu_k) 
\end{equation} 
to denote a path $\omega$ in graph $E_\J$ with consecutive edges $\omega_j=[(\mu_j,e_j\nu_j),
(\mu_{j+1},e_{j+1}\nu_{j+1})]$, for $j=1,2,\ldots,k-1$. The labelling $\LL_\J$ extends to a map 
$\LL_\J \colon E_\J^*\to L_\Zb(E)$ so that if $\omega\in E_\J^*$ is as in (\ref{pathinej}) then 
$$ \LL_\J(\omega) = \LL_\J(\omega_1)\cdots\LL_\J(\omega_{k-1}). $$
If the path $\omega$ consists of a single vertex $(\mu,e\nu)$, then we put $\LL_\J(\omega)=P_{r(\mu)}$. Furthermore, 
we define two maps $\LL_s$ and $\LL_r$, both from $E_\J^*$ to $\{S_\alpha \mid \alpha\in E^*\}\cup
\{P_v \mid v\in E^0\}$, so that 
$$ \begin{aligned}
\LL_s(\omega) & = S_{\mu_1}, \\
\LL_r(\omega) & = S_{\nu_k}. 
\end{aligned} $$
We also define a map $\E \colon E_\J^* \to E^*$ by 
$$ \E(\omega) = e_1 e_2 \ldots e_k. $$
That is, $\E(\omega)$ is a list of $e$ from all the vertices $\omega$ visits. 
Note that this mean that if $\omega$ has length $k$, then $\E(\omega)$ has length $k+1$.
Thus if we have two paths $\omega$ and $\xi$ with $r(\omega) = s(\xi)$, then it is not the case that $\E(\omega\xi) = \E(\omega)\E(\xi)$ since on the right hand side the $e$ of the vertex $r(\omega)$ appears both at the end of $\E(\omega)$ and at the start of $\E(\xi)$.

The following lemma illustrates the usefulness of coding graphs for dealing with endomorphisms and essentially 
motivates the very definition of a coding graph. 

\begin{lemma}\label{imageofpath}
Let $u_\J \in \SS_E$ and let $\alpha\in E^*$. 
Then 
\[
 \Lambda_J(S_\alpha)  = \sum_{\E(\omega)=\alpha}\LL_s(\omega)\LL_\J(\omega)\LL_r(\omega)^* 
\]
\end{lemma}
\begin{proof}
If $|\alpha|=1$ and thus $\alpha$ consits of a single edge $e$, 
then $\E(\omega)=e$ if and only if $\omega\in E_\J^0$ and $\omega=(\mu,e\nu)$. Hence 
\[
 \Lambda_\J(S_e) = u_\J S_e = \sum_{(\mu,e\nu)\in E_\J^0} S_\mu P_{r(\mu)} S_\nu^* = 
\sum_{\E(\omega)=e}\LL_s(\omega)\LL_\J(\omega)\LL_r(\omega)^*. 
\]
If $|\alpha|=k\geq2$ and $\alpha=e_1e_2\ldots e_k$, then we have 
\begin{align*}
\Lambda_\J(S_\alpha) & = \sum_{(\mu_1,e_1\nu_1)\in E_\J^0} \ldots \sum_{(\mu_k,e_k\nu_k)\in E_\J^0} 
(S_{\mu_1}S_{\nu_1}^*)\cdots(S_{\mu_k}S_{\mu_k}^*) \\
 & = \sum_{(\mu_1,e_1\nu_1)\in E_\J^0} \ldots \sum_{(\mu_k,e_k\nu_k)\in E_\J^0} 
S_{\mu_1}(S_{\nu_1}^*S_{\mu_2})\cdots(S_{\nu_{k-1}}^*S_{\mu_k})S_{\nu_k}^* \\
 & = \sum_{\E(\omega)=\alpha}\LL_s(\omega)\LL_\J(\omega)\LL_r(\omega)^*.  \qedhere
\end{align*}
\end{proof}

\begin{corollary}\label{imageofprojpath}
Let $\mu\in\J_1$. 
Given a $\delta\in E^*$, consider the (possibly empty) collection of all paths $\{A_i\}$ in $E_\J^*$ with 
$s_\J(A_i)=J_\mu$ and $\E(A_i)=\delta$. Then 
\[
 P_\mu \Lambda_\J(S_\delta) = \sum_i S_\mu\LL_J(A_i)\LL_r(A_i)^*. 
\]
\end{corollary}

Let $\eta\in E_\J^1$ be an edge in the coding graph $E_\J$. We define its {\em degree} as follows:
$$ \deg(\eta) = \left\{\hspace{-2mm} \begin{array}{rl} |\alpha| & \text{if}\; \LL_\J(\eta)=S_\alpha, \\ 
0 & \text{if}\; \LL_\J(\eta)\in\D_E^0, \\ -|\alpha| & \text{if}\; \LL_\J(\eta)=S_\alpha^*. \end{array} \right. $$

We say that a path in $E_\J$ is {\em positive}, {\em non-negative}, {\em negative} or {\em non-positive}, when 
the degrees of all of its edges are positive, non-negative, negative or non-positive, respectively. 
A {\em zero} path consists entirely of edges with degree $0$.  

\begin{lemma}\label{outedges}
Let $J_\mu\in E_\J^0$. Then the following holds. 
\begin{description}
\item{(i)} If the vertex $J_\mu$ emits a non-positive edge then it emits no other edges. 
\item{(ii)} If the vertex $J_\mu$ emits a positive edge then for each edge $[J_\mu,J_\delta]$ it emits there is a path $\alpha\in E^*$ 
such that $\LL_\J([J_\mu,J_\delta])=S_\alpha$, and the collection of all these paths constitutes a partition of $r(\mu)$. 
\end{description}
\end{lemma}
\begin{proof}
Ad (i). Let $S_\alpha^*=S_\nu^*S_{\mu_1}=\LL_\J([(\mu,e\nu),J_{\mu_1}]$, and suppose that $[(\mu,e\nu),J_{\mu_2}]$ 
is another edge emitted by $J_\mu$. Then $0\neq S_\nu^*S_{\mu_2}=S_\nu^*S_{\mu_1}S_{\mu_1}^*S_{\mu_2}$. Thus 
$P_{\mu_1}P_{\mu_2}\neq 0$, contradicting (\ref{twopartitions}). 

\vspace{2mm}\noindent
Ad (ii). Let $J_{\delta_j}$, $j=1,\ldots,m$, be the ranges of all edges emitted by vertex $J_\nu=(\mu,e\nu)$. Since $J_\mu$ 
emits a positive edge, part (i) of this lemma implies that all edges $[J_\mu,J_{\delta_j}]$ are positive. Thus, there exist paths 
$\alpha_j\in E^*$ such that $S_\nu^*S_{\delta_j}=S_{\alpha_j}$. Thus $\delta_j=\nu\alpha_j$, and hence
$\{\alpha_j \mid j=1,\ldots,m\}$ constitutes a partition of $r(\mu)=r(\nu)$ by Lemma \ref{partitions}. 
\end{proof}

\begin{corollary}\label{rresolving}
The labeled graph $(E_\J,\LL_\J)$ is \emph{right-resolving}.
That is, none of its vertices emits two distinct edges with the same label $\LL_J$. 
\end{corollary} 

\begin{corollary}\label{cycleswithoutexits}
A cycle in graph $(E_\J,\LL_\J)$ has no exits if and only if it is non-positive. 
\end{corollary}

The following corollary of Lemma \ref{outedges} will be needed later in Section \ref{autodiag}. 

\begin{corollary}\label{diagpart}
Let $\alpha\gamma\in E^*$ with $\alpha\in\J_1$. For an arbitrarily large inetger $l\in\Nb$ there exist paths $A_1,\ldots,A_m\in E_\J^*$ with 
$s_\J(A_i)=J_\alpha$, $|\LL_\J(A_i)|\geq l$, and $\LL_\J(A_i)=S_{\gamma\omega_i}$ with $\{\omega_i\}\subseteq E^*$ such that 
$$ P_{r(\gamma)} = \sum_i P_{\omega_i}. $$ 
\end{corollary}


\section{Splitting}

If $u_\J$ is a {\em permutative} unitary, that is if $u_\J\in\SS_E\cap\F_{L_\Zb(E)}$, then its coding graph takes a particularly 
simple form. Namely, the $\LL_\J$ label of each edge in $E_\J^1$ is $S_f$ for some $f\in E^1$. It is largely due to 
this fact that permutative endomorphisms are relatively easier to analyze, as was done in \cite{AJSz, CHSjfa, CStams}.
Unfortunately, a general coding graph $E_\J$ may contain a variety of positive and negative edges. 
Dealing with such graphs and thus analyzing the corresponding endomorphisms is substantially more involved. 
To alleviate this problem, in the present section we discuss a procedure called {\em splitting} for rewriting a unitary 
$u_\J$ in such a way that the corresponding coding graph either has only non-negative edges, or else admits a 
non-positive cycle. In the former case, dealing with the coding graph $E_\J$ becomes much easier, see section 
\ref{MASA}  below. In the latter case, endomorphism $\Lambda_\J$ does not restrict to an 
automorphism of $\D_{L_{\Zb(E)}}$, see Lemma \ref{starcyclesmasa} below. 
Hence $\lambda_\J$ cannot restrict to an automorphism of $\D_{C^*(E)}$ (Lemma \ref{masaisoiff}).

\begin{definition}\label{splittingdef}
Let $u_\J$ be a unitary in $\SS_E$ and let $E_\J$ be its coding graph. The {\em splitting} of $E_\J$ at vertex 
$(\mu,e\nu)\in E_\J^0$ is the coding graph $E_{\J'}$, where 
$$ \J' = (\J\setminus\{(\mu,e\nu)\}) \cup \{(\mu f,e\nu f) \mid f\in E^1,\,s(f)=r(\mu)\}. $$
\end{definition}
Note that $u_{\J'}=u_\J$, so that every splitting of $E_\J$ constitutes another coding graph for the same unitary $u_\J$. 
We will only perform splittings at vertices which emit only positive edges, and from now on this is always assumed. 

If $\eta=[(\mu_0,e_0\nu_0),(\mu_1,e_1\nu_1)]$ is a negative edge in $E_\J^1$ and there is a zero path 
$(\mu_1,e_1\nu_1)\Rightarrow(\mu_k,e_k\nu_k)$, possibly of length $0$,  such that its range $(\mu_k,e_k\nu_k)$ 
emits only positive edges, then edge $\eta$ will be called {\em final negative}. The vertex $(\mu_k,e_k\nu_k)$ will be 
called the {\em destination} of $\eta$ and denoted $d(\eta)$.  The length of the zero path emitted by a final negative edge 
$\eta$ will be called {\em height} of $\eta$ and denoted $h_\J(\eta)$. 

Consider a splitting described in Definition \ref{splittingdef}. 
Edges between vertices different from $(\mu,e\nu)$ are unaffected by the splitting. 
New edges in graph $E_{\J'}$ created by the splitting may be described as follows. 
\begin{description}
\item{(SE1)} If the graph $E_\J$ contains an edge $\eta=[(\mu,e\nu),(\delta,g\gamma)]$ with $\mu\neq\delta$, 
then there is exactly one $f\in E^1$, $s(f)=r(\mu)$, such that the graph $E_{\J'}$ contains an edge $\eta'=
[(\mu f,e\nu f),(\delta,g\gamma)]$. We have $\deg(\eta')=\deg(\eta)-1$. 
\item{(SE2)} If the graph $E_\J$ contains an edge $\eta=[(\mu,e\nu),(\mu,e\nu)]$, 
then there is exactly one $f\in E^1$, $s(f)=r(\mu)$, such that the graph $E_{\J'}$ contains an edge $\eta'_b=
[(\mu f,e\nu f),(\mu b,e\nu b)]$ for each $b\in E^1$, $s(b)=r(\mu)$. We have $\deg(\eta'_b)=\deg(\eta)$. 
\item{(SE3)} If the graph $E_\J$ contains a non-negative edge $\eta=[(\delta,g\gamma),(\mu,e\nu)]$ with $\mu\neq\delta$,  
then for each $f\in E^1$, $s(f)=r(\mu)$, the graph $E_{\J'}$ contains  an edge $\eta'_f=[(\delta,g\gamma),(\mu f,e\nu f)]$. 
We have $\deg(\eta'_f)=\deg(\eta)+1$. 
\item{(SE4)} If the graph $E_\J$ contains a negative edge $\eta=[(\delta,g\gamma),(\mu,e\nu)]$ with $\mu\neq\delta$,  
then there is exactly one $f\in E^1$, $s(f)=r(\mu)$, such that the graph $E_{\J'}$ contains an edge $\eta'=
[(\delta,g\gamma),(\mu f,e\nu f)]$. We have $\deg(\eta')=\deg(\eta)+1$. The edge $\eta'\in E_{\J'}^1$ will be called the 
{\em descendant} of the edge $\eta\in E_\J^1$. 
\end{description}

We are now going to describe the {\em splitting algorithm} for transforming coding graphs. 

\begin{definition}\label{splittingalgorithm}
Let $u_\J\in\SS_E$. If the graph $(E_\J,\LL_\J)$ either 
\begin{description}
\item{(i)} contains a non-positive cycle, or 
\item{(ii)} all its edges are non-negative 
\end{description}
then the algorithm ends. Otherwise, perform a splitting at the destination of a final negative edge with the lowest height. 
\end{definition}

If a splitting is performed at the destination of a final negative edge $\eta$ then this edge will be called {\em active}. Note 
that there may be more than one active edge. 

Applicability of the splitting algorithm is justified by the following theorem. 

\begin{theorem}\label{splittingtheorem}
For each $u_\J\in\SS_E$, the splitting algorithm described in Definition \ref{splittingalgorithm} 
terminates after finitely many steps.  
\end{theorem}
\begin{proof}
We assume that the graph $(E_\J,\LL_\J)$ contains some negative edges and does not contain non-positive cycles. 
It follows immediately from (ES1)--(SE4) that splitting does not increase the number of negative edges. Thus, it suffices 
to show that sufficiently many splittings will eventually decrease this number. For this to happen, we must show that 
a sufficiently long sequence of splittings will create a final negative edge with height $0$ and degree $-1$. 

Now, if the graph $E_\J$ contains a final negative edge with height $0$, then the splitting will be performed at such a vertex 
(there may be more than one), and will result in increasing the degree of each of the active edges with height $0$. Thus, 
eventually, one of these edges will have degree $-1$. 

On the other hand, suppose that heights of all final negative edges are positive. It follows from (SE1)--(SE4) that 
final negative edges in $E_\J^1$ will remain so in $E_{\J'}^1$ after one application of splitting. The heights of active edges 
will decrease by $1$, while the heights of inactive ones will remain unchanged. Thus, after a sufficiently long sequence of 
splittings, there will be created a final negative edge with height $0$. 
\end{proof}


\section{Automorphisms of the diagonal MASA}\label{MASA}\label{autodiag}

\begin{lemma}\label{starcyclesmasa}
Let $u_\J\in\SS_E$. 
If its coding graph $E_\J$ contains a non-positive cycle then the endomorphism $\Lambda_\J$ does not restrict to an automorphism of the diagonal. 
That is, $\Lambda_\J(\D_{L_\Zb(E)})$ is a proper subset of $\D_{L_\Zb(E)}$ and consequently $\Lambda_\J\not\in\aut(L_\Zb(E))$. 
\end{lemma}
\begin{proof}
Let $(\mu_1,e_1\nu_1)\Rightarrow(\mu_k,e_k\nu_k)$ be a non-positive cycle. 
We denote it $\eta$, and for each $r=2,\ldots,k$ we denote the initial subpath $(\mu_1,e_1\nu_1)\Rightarrow(\mu_r,e_r\nu_r)$ of $\eta$ by $\eta_r$. 
By Corollary \ref{cycleswithoutexits}, the cycle $\eta$ has no exit. 
By Lemma \ref{imageofpath}, for each $\alpha\in E^*$ we have 
\[
	P_{\mu_1}\Lambda_\J(S_\alpha) = \sum_{\E(\omega)=\alpha}P_{\mu_1}\LL_s(\omega)\LL_\J(\omega)\LL_r(\omega)^*. 
\]
If this sum is non-zero then $\LL_s(\omega)=\mu_1$, and hence $s(\omega)=(\mu_1,e_1\nu_1)$. 
Thus $\omega = \eta\ldots\eta\eta_r$ for some $r$, with initial subpath $\eta\ldots\eta$ consisting of the cycle $\eta$ repeated $m$ times (possibly $m=0$). 
Consequently, $P_{\mu_1}\lambda_\J(\D_{L_\Zb(E)})$ is the closed span of elements of the form 
\[
	S_{\mu_1}\LL_\J(\eta)^m\LL_J(\eta_r)\LL_r(\eta_r)^*\LL_r(\eta_r) \LL_\J(\eta_r)^*\LL_J(\eta)^{*m}S_{\mu_1}^*. 
\]
Since the cycle $\eta$ is non-positive, each such an element is either $0$ or equal to $P_{\mu_1}$.
Thus $P_{\mu_1}\Lambda_\J(\D_{L_\Zb(E)})=\Cb P_{\mu_1}$ and $\Lambda_\J(\D_{L_\Zb(E)})\neq\D_{L_\Zb(E)}$.  
\end{proof}

Taking into account Theorem \ref{splittingtheorem} and Lemma \ref{starcyclesmasa} above, we may from now on restrict our attention to those $u_\J\in\SS_E$ whose corresponding coding graphs $E_\J$ have {\bf only non-negative edges and contain no non-positive cycles}. 
In this setting we get some nice extra properties of the coding graph.

\begin{lemma} \label{lresolving}
Let $u_\J\in\SS_E$. 
If the coding graph $E_\J$ has only non-negative edges, then $E_\J$ is \emph{left-resolving} in the $\E$-label. 
That is, if $\omega,\xi \in E_\J^1$ are such that $\E(\omega) = \E(\xi)$ and $r(\omega) = r(\xi)$, then $\omega = \xi$.
\end{lemma}
\begin{proof}
Let 
\[
	r(\omega) = \vertexbox{e}{\mu}{\nu} = r(\xi), \quad s(\omega) = \vertexbox{f_1}{\mu_1}{\nu_1}\text{ ,} \quad \text{ and } s(\xi) = \vertexbox{f_2}{\mu_2}{\nu_2} \text{ }.
\]
Since $E_\J$ does not have multiple edges, it suffices to prove that $s(\omega) = s(\xi)$, which will follow if $f_1 \nu_1 = f_2 \nu_2$.
Let $f = f_1$.
From $\E(\omega) = \E(\xi)$ we get that $f_1 = f_2$. 
From the fact that the edge $\omega$ exist and that all edges in $E_\J$ are non-negative we get that $\nu_1 \prec \mu$.
Similarly $\nu_2 \prec \mu$. 
Therefore we have that either $\nu_1 \prec \nu_2$ or $\nu_2 \prec \nu_1$. 
So either $f \nu_1 \prec f \nu_2$ or $f \nu_2 \prec f \nu_1$.
Since $u_\J \in \SS_E$ this is only possible if $f_1 \nu_1 = f \nu_2 = f_2 \nu_2$.
\end{proof} 

Note that the left-resolving property extends to paths.
That is, if $\omega$ and $\xi$ are paths with $\E(\omega) = \E(\xi)$ and $r(\omega) = r(\xi)$, then $\omega = \xi$. 
To see this, one first observes that left-resolving implies that the final edges of $\omega$ and $\xi$ must agree, and then one simply works backward through the path. 

The following strengthening of Lemma \ref{outedges}, comes simply by noticing that Lemma \ref{outedges} applies to all vertices in the case of no negative edges.

\begin{lemma} \label{outpaths}
Let $u_\J\in\SS_E$. 
We assume that the coding graph $E_\J$ has only non-negative edges and does not contain any non-positive cycles. 
For any vertex $J_\mu \in E_\J^0$ and any $l \in \Nb$, we have that
\[
	\setof{\LL_\J(\omega)}{\omega \in E_\J^l, s(\omega) = J_\mu }
\]
is a partition of $r(\mu)$.
\end{lemma}

We now define a property of the coding graph that will turn out to be equivalent to $\Lambda_u$ restricting to an automorphism of the diagonal. 

\begin{definition}\label{leftsync}
We say that the labeled graph $(E_\J,\E)$ is {\em left-synchronizing} (with delay $m$) if there exists an $m\in\Nb$ such that for any two paths $\omega_1,\omega_2\in E_\J^m$ of length $m$ if $\E(\omega_1)=\E(\omega_2)$ then $s(\omega_1)=s(\omega_2)$. 
\end{definition}

We first show that a left-synchronizing coding graph implies that the endomorphism restricts to an isomorphism on $\D_{L_\Zb(E)}$. 
A similar criterion for endomorphisms of the Cuntz algebras $\OO_n$ has been found earlier in \cite{CHSam}.  

\begin{theorem}\label{diagonalauto}
Let $u_\J\in\SS_E$. 
We assume that the coding graph $E_\J$ has only non-negative edges and does not contain any non-positive cycles. 
If the labeled graph $(E_\J,\E)$ is left-synchronizing then the endomorphism $\Lambda_\J \colon L_{\Zb}(E) \to L_{\Zb}(E)$ 
restricts to automorphism of the MASA $\D_{L_\Zb(E)}$. 
\end{theorem}
\begin{proof}
Assume that graph $(E_\J,\E)$ is left-synchronizing with delay $l$. In order to show that $\Lambda_\J|_{\D_{L_\Zb(E)}}$ is an automorphism of $\D_{L_\Zb(E)}$, we must show that all projections $P_\mu$, $\mu\in E^*$, are in the range of $\Lambda_\J|_{\D_{L_\Zb(E)}}$. 
It suffices to do this for all paths $\mu$ of length exceeding an arbitrary fixed integer, since every projection in $\D_{L_\Zb(E)}$ may be written as a sum of such projections $P_\mu$. 

Now, consider  a path $\mu\in E^*$ that is sufficiently long to have a prefix $\mu_p\in\J_1$, and let $\mu_s\in E^*$ be the unique path such that 
$\mu = \mu_p\mu_s$. By Corollary \ref{diagpart}, the vertex $J_{\mu_p}\in E_\J$ emits paths $\{A_i\}\subseteq E_\J^*$ such that 
$\LL_\J(A_i)=S_{\mu_s\omega_i}$ and $P_{r(\mu)} = \sum_i P_{\omega_i}$. We put $J_i = r_\J(A_i)$. Let $\{\gamma_{ij}\}$ be the 
collection of $\E$-labels of paths in $E_\J^*$ of length $l$ and source $J_i$. Left-synchronization of the graph $(E_\J, \E)$ implies that 
every path with $\E$-label $\gamma_{ij}$ must have $J_i$ as its source.  For each $\gamma_{ij}$ there may be more than one path with 
$\E$-label $\gamma_{ij}$, but then the $\LL_\J$-labels uniquely identified such paths (Lemma \ref{rresolving}). 

Let $R_{ijm}$ be the collection of all $\LL_\J$-labels of all paths with $\E$-label $\gamma_{ij}$. By Lemma \ref{outedges}, the ranges of $R_{ijm}$ 
yield a partition of $P_{r(\omega_i)}$. Let $J_{ijm}\in E_\J^0$ be the range of the unique path with the double $\E$-$\LL_\J$-label $[\gamma_{ij},R_{ijm}]$. 
Note that $J_{ijm}\neq J_{ijm'}$ for $m\neq m'$, since the coding graph ls left-resolving in the $\E$-label (Lemma \ref{lresolving}). Furthermore, since the graph 
$(E_\J,\E)$ is both left-resolving and left-synchronizing, every path with $\E$-label $\E(A_i)\gamma_{ij}$ must have $S_{\mu_s\omega_i}$ as a prefix 
of its $\LL_\J$ label. We denote by $K_{ijm}$ the $\LL_r$-label of the unique path with the double label $[\gamma_{ij},R_{ijm}]$. 
By Lemma \ref{imageofpath}, we now get 
\begin{align*}
\Lambda_\J\left(\sum_{i,j} S_{\E(A_i)}S_{\gamma_{ij}}S_{\gamma_{ij}}^*S_{\E(A_i)}^*\right) & = 
S_{\mu_p}S_{\mu_s}\left(\sum_{i,j,m,m'} S_{\omega_i} R_{ijm} K_{ijm}^* K_{ijm'} R_{ijm'}^*                                 
S_{\omega_i}^*\right)S_{\mu_s}^*S_{\mu_p}^* \\
 & = S_\mu \left(\sum_{i,j,m} S_{\omega_i} R_{ijm}R_{ijm'}^*S_{\omega_i}^*\right)S_\mu^* \\
 & = P_\mu. 
\end{align*}

The following picture illustrates the above process for a single $\gamma_{ij}$. 
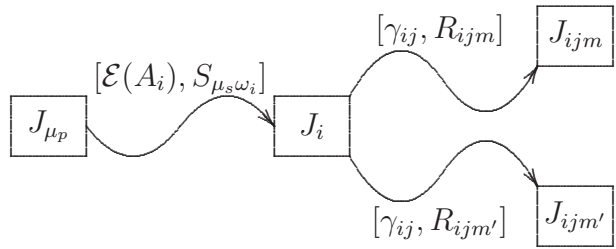
\[ \beginpicture
\setcoordinatesystem units <1cm,0.8cm>
\setplotarea x from -4 to 4, y from -2 to 2
\setlinear
\plot -4 0.5  -3 0.5  -3 -0.5  -4 -0.5  -4 0.5 /
\plot -0.5 0.5  0.5 0.5  0.5 -0.5  -0.5 -0.5  -0.5 0.5 /
\plot 3 2  4 2  4 1  3 1  3 2 /  
\plot  3 -1  4 -1  4 -2  3 -2  3 -1 /  
\setquadratic
\plot  -3 0  -2.375 -0.5  -1.75 0 /
\plot  -1.75 0  -1.125 0.5  -0.5 0 /
\plot  0.5 0.5  1.125 1.25  1.75 0.75 /
\plot  1.75 0.75  2.375 0.25  3 1 /
\plot  1.75 -0.75  2.375 -0.25  3 -1 /
\plot  0.5 -0.5  1.125 -1.25  1.75 -0.75 /
\put {$J_{\mu_p}$} at -3.5 0
\put {$J_i$} at 0 0
\put {$J_{ijm}$} at 3.5 1.5
\put {$J_{ijm'}$} at 3.5 -1.5 
\put {$[\E(A_i),S_{\mu_s\omega_i}]$} at -1.75 0.8
\put {$[\gamma_{ij},R_{ijm}]$} at 1.7 1.6
\put {$[\gamma_{ij},R_{ijm'}]$} at 1.7 -1.6
\arrow <0.235cm> [0.2,0.6] from -0.65 0.2 to -0.5 0
\arrow <0.235cm> [0.2,0.6] from 2.9 0.8   to 3 1
\arrow <0.235cm> [0.2,0.6] from  2.9 -0.8  to 3 -1
\endpicture \]
\end{proof}

To prove that the coding graph being left-synchronizing is also a necessary condition we record some more consequences of $E_\J$ being left resolving in the $\E$-label.

\begin{lemma} \label{lem:twocycles}
Let $u_\J\in\SS_E$. 
We assume that the coding graph $E_\J$ has only non-negative edges and does not contain any non-positive cycles. 
If $E_\J$ is not left-synchronizing in the $\E$-label, then there exists two distinct cycles in $E_\J^*$ with the same $\E$-label.
\end{lemma}
\begin{proof}
As in $(3) \implies (2)$ in the proof of \cite[Lemma 3.10]{AJSz}.
\end{proof}

The following Lemma is just a minor variation on Lemma \ref{lresolving}.

\begin{lemma} \label{lem:rangelabels}
Let $u_\J\in\SS_E$. 
We assume that the coding graph $E_\J$ has only non-negative edges and does not contain any non-positive cycles. 
Suppose that $\omega, \xi \in E_\J^*$ and $\E(\omega) = \E(\xi)$. 
Then $\LL_r(\omega)^* \LL_r(\xi) \neq 0$ if and only if $\omega = \xi$. 
Furthermore, for any path $\gamma \in E_\J^*$ we have
\[
	\LL_\J(\gamma) \LL_r(\gamma)^* \LL_r(\gamma) = \LL_\J(\gamma).
\]
\end{lemma}
\begin{proof}
Let
\[
	r(\omega) = \vertexbox{e_1}{\mu_1}{\nu_1} \quad \text{ and } \quad
	r(\xi) = \vertexbox{e_2}{\mu_2}{\nu_2} \text{ } .
\]
Since $\E(\omega) = \E(\xi)$, we see that $e_1 = e_2$. 
For simplicity put $e = e_1$. 
We have
\begin{align*}
	\LL_r(\omega)^* \LL_r(\xi) \neq 0 &\iff \nu_1 \prec \nu_2 \text{ or } \nu_2 \prec \nu_1 \\
	&\iff e\nu_1 \prec e\nu_2 \text{ or } e\nu_2 \prec e\nu_1 \\
	&\iff e\nu_1 = e\nu_2 \\
	&\iff r(\omega) = r(\xi) \\
	&\iff \omega = \xi.
\end{align*}
Where the last equivalence follows since $E_\J$ is left-resolving.

It suffices to prove the final claim of the lemma under the assumption that $\gamma$ is an edge. 
So suppose $\gamma$ is an edge with 
\[
	s(\gamma) = \vertexbox{e_1}{\mu_1}{\nu_1} \quad \text{ and } \quad
	r(\gamma) = \vertexbox{e_2}{\mu_2}{\nu_2} \text{ }.
\] 
Then $\LL_r(\gamma) = S_{\nu_2}$ and $\LL_\J(\gamma) = S_{\nu_1}^* S_{\mu_2}$. 
Hence 
\[
	\LL_r(\gamma)^* \LL_r(\gamma) = S_{\nu_2}^* S_{\nu_2} = P_{r(\nu_2)} = P_{r(\mu_2)}.
\]
Therefore 
\[
	\LL_\J(\gamma) \LL_r(\gamma)^* \LL_r(\gamma) = S_{\nu_1}^* S_{\mu_2} P_{r(\mu_2)} = S_{\nu_1}^* S_{\mu_2} = \LL_\J(\gamma). \qedhere
\]
\end{proof}

\begin{theorem} \label{tmm:notdiagiso}
Let $u_\J\in\SS_E$. 
We assume that the coding graph $E_\J$ has only non-negative edges and does not contain any non-positive cycles. 
If $E_\J$ is not left-synchronizing in the $\E$ label, then $\Lambda_\J$ does not restrict to an isomorphism of $\D_{L_\Zb(E)}$.
\end{theorem}
\begin{proof}
By Lemma \ref{lem:twocycles} there exists two distinct cycles $\beta, \beta' \in E_\J^*$ with $\E(\beta) = \E(\beta')$. 
Since $E_\J$ is left-resolving in the $\E$-label (Lemma \ref{lresolving}) $\beta$ and $\beta'$ have distinct ranges. 
I.e. there exists distinct vertices $J_\mu, J_{\mu'} \in E_\J^0$ such that $s(\beta) = r(\beta) = J_\mu$ and $s(\beta') = r(\beta') = J_{\mu'}$. 
We claim that $P_\mu$ is not in the image of $\Lambda_\J$.

Suppose, for contradiction, that $P_\mu$ is in the image of $\Lambda_\J$. 
Since $\Lambda_\J$ maps $\D_{L_\Zb(E)}$ into $\D_{L_\Zb(E)}$, this means there exists a finite set of paths $\{\alpha_i\}$ such that
\[
	P_\mu = \Lambda_\J \left( \sum_i S_{\alpha_i} S_{\alpha_i}^* \right).
\]
Since $E$ has no sinks or sources, we can assume that all the $\alpha_i$ have the same length, $l$ say.
Using lemmas \ref{imageofpath} and \ref{lem:rangelabels} we compute.
\begin{align}
\begin{split} \label{eq:pmu}
	P_\mu &= \Lambda_\J \left( \sum_i S_{\alpha_i} S_{\alpha_i}^* \right) = 
	\sum_i \Lambda_\J( S_{\alpha_i}) \Lambda_\J(S_{\alpha_i}^*) \\
	&= \sum_i \left( \sum_{\E(\omega) = \alpha_i} \LL_s(\omega) \LL_\J(\omega) \LL_r(\omega)^* \right) \left( \sum_{\E(\xi) = \alpha_i} \LL_r(\xi) \LL_\J(\xi)^* \LL_s(\xi)^* \right) \\
	&= \sum_i \left( \sum_{\E(\omega) = \alpha_i = \E(\xi)} \LL_s(\omega) \LL_\J(\omega) \LL_r(\omega)^* \LL_r(\xi) \LL_\J(\xi)^* \LL_s(\xi)^* \right) \\
	&= \sum_i \left( \sum_{\E(\omega) = \alpha_i} \LL_s(\omega) \LL_\J(\omega) \LL_\J(\omega)^* \LL_s(\omega)^* \right)
\end{split}
\end{align}
We will use this equation to first establish that there exists some $k$ such that $\alpha_k \prec \E(\beta)^\infty$.
Then we will leverage that to get a contradiction. 

Multiplying Equation (\ref{eq:pmu}) by $P_\mu$ we then get (as in Lemma \ref{imageofprojpath}) the following
\begin{align*}
	P_\mu &= P_\mu^2 \\
	&= \sum_i \left( \sum_{\E(\omega) = \alpha_i} P_\mu \LL_s(\omega) \LL_\J(\omega) \LL_\J(\omega)^* \LL_s(\omega)^* \right) \\
	&= \sum_i \left( \sum_{\substack{\E(\omega) = \alpha_i \\ s(\omega) = J_\mu}} \LL_s(\omega) \LL_\J(\omega) \LL_\J(\omega)^* \LL_s(\omega)^* \right) \\
	&= S_{\mu} \left( \sum_i \left( \sum_{\substack{\E(\omega) = \alpha_i \\ s(\omega) = J_\mu}} \LL_\J(\omega) \LL_\J(\omega)^* \right) \right) S_{\mu}^* \\
	&= S_{\mu} \left( \sum_{\substack{\E(\omega) \in \{\alpha_i\} \\ s(\omega) = J_\mu}} \LL_\J(\omega) \LL_\J(\omega)^* \right) S_{\mu}^* 
\end{align*}
Therefore $\setof{\LL_\J(\omega)}{s(\omega) = J_\mu \text{ and } \E(\omega) = \alpha_i \text{ for some } i}$ is a partition of $P_{\mu}$.
Comparing with Lemma \ref{outpaths} this tells us that any path $\omega \in E_\J^l$ with $s(\omega) = J_\mu$ must have $\E(\omega) = \alpha_i$ for some  $i$. 
In particular there exists a $k$ such that $\alpha_k \prec \E(\beta)^\infty$.

Multiplying equation (\ref{eq:pmu}) by $P_{\mu'}$ and computing as above we get
\begin{align*}
	0 &= P_{\mu'} P_{\mu} \\
	&= \sum_i \left( \sum_{\substack{\E(\omega) = \alpha_i \\ s(\omega) = J_\mu}} \LL_s(\omega) \LL_\J(\omega) \LL_\J(\omega)^* \LL_s(\omega)^* \right)
\end{align*}
Since we are working in a Leavitt path algebra over $\Zb$ there is no cancellation in this sum, so it can only be $0$ if there does not exist a path $\omega$ with $s(\omega) = J_{\mu'}$ and $\E(\omega) = \alpha_i$.
However, as we saw above there exists a $k$ such that $\alpha_k \prec \E(\beta)^\infty = \E(\beta')^\infty$, so there exists a path $\xi \prec (\beta')^\infty$ such that $\E(\xi) = \alpha_k$ and $s(\xi) = J_{\mu'}$.
\end{proof}

We summarize our results and use Lemma \ref{masaisoiff} to translate them into the language of $C^*$-algebras. 

\begin{corollary} \label{cor:diagonalauto}
Let $u_\J\in\SS_E$. 
\begin{enumerate}
	\item If the coding graph $E_\J$ contains a non-positive cycle then the endomorphism $\Lambda_\J$ does not restrict to an automorphism of the diagonal.
	\item Suppose the coding graph $E_\J$ has only non-negative edges and does not contain any non-positive cycles. 
	The labeled graph $(E_\J,\E)$ is left-synchronizing if and only if the endomorphism $\Lambda_\J \colon L_{\Zb}(E) \to L_{\Zb}(E)$ restricts to automorphism of the MASA $\D_{L_{\Zb}(E)}$.    
	\item If the coding graph $E_\J$ contains a non-positive cycle then the endomorphism $\lambda_\J$ does not restrict to an automorphism of the diagonal.
	\item Suppose the coding graph $E_\J$ has only non-negative edges and does not contain any non-positive cycles. 
	The labeled graph $(E_\J,\E)$ is left-synchronizing if and only if the endomorphism $\lambda_\J \colon C^*(E) \to C^*(E)$ restricts to automorphism of the MASA $\D_{C^*(E)}$.    
\end{enumerate}
\end{corollary}

We end this section by taking a small step towards the goal of establishing when $\Lambda_u$ is onto, otherwise not discussed in the present paper. 
Consider a path $\alpha \in E^*$, by Lemma \ref{imageofpath} we can compute $\Lambda_u(S_\alpha)$ as
\[
	\Lambda_u(S_\alpha) = \sum_{\E(\omega) = \alpha} \LL_s(\omega) \LL_\J(\omega) \LL_r(\omega)^*.
\]
The purpose of the next three lemmas is to show that if $\Lambda_u$ restricts to in isomorphism of the diagonal and the coding graph has no non-negative edges, then in fact each summand is in the image of $\Lambda_u$. 

We record the following lemma mostly to fix notation. 

\begin{lemma}
Let $u_\J \in \SS_E$ be written 
\[
	u_\J = \sum_{i=1}^n S_{\mu_i} S_{\nu_i}^*.
\]
Assume that the coding graph $E_\J$ has only non-negative edges and does not contain any non-positive cycles. 
For each path $\omega \in E_\J^*$ there is a $k$ and a path $\gamma \in E^*$ such that
\[
	\LL_s(\omega) = S_{\mu_k}, \quad \text{ and } \quad \LL_\J(\omega) = s_\gamma. 
\]
\end{lemma}

\begin{lemma} \label{lem:nosubpathlabel}
Let $u_\J \in \SS_E$ be written 
\[
	u_\J = \sum_{i=1}^n S_{\mu_i} S_{\nu_i}^*.
\]
Assume that the coding graph $E_\J$ has only non-negative edges and does not contain any non-positive cycles. 

Let $\omega, \xi \in E_\J^*$ be paths of the same length, and write
\begin{align*}
	\LL_s(\omega) \LL_\J(\omega) &= S_{\mu_k} S_{\gamma}, \\
	\LL_s(\xi) \LL_\J(\xi) &= S_{\mu_l} S_{\delta},
\end{align*}
for $l,k \in \{1,2, \ldots, n\}$ and paths $\gamma, \delta \in E^*$. 
If $\mu_k \gamma \prec \mu_l \delta$ then $\omega = \xi$
\end{lemma}
\begin{proof}
We prove the claim by induction on the length $m = |\omega| = |\gamma|$.
If $m = 0$ then $\LL_\J(\omega),\LL_\J(\xi)$ are vertex projections, so $\gamma$ and $\delta$ are both vertices. 
Hence our assumption becomes that $\mu_k \prec \mu_l$.
By equation (\ref{twopartitions}) this is only possible if $k = l$, so we get
\[
	\omega = J_{\mu_k} = J_{\mu_l}= \xi.  
\]

Suppose now $m > 0$ and that the claim is proved for pairs of paths of length less than $m$.
Since $E_\J$ only has non-negative edges we can find paths $\gamma_i, \delta_i \in E^*$, $i = 1, 2, \ldots, m$ such that $\LL_\J(\omega_i) = S_{\gamma_i}$ and $\LL_\J(\xi_i) = S_{\delta_i}$. 
Note that $\gamma = \gamma_1 \gamma_2 \cdots \gamma_m$ and $\delta = \delta_1 \delta_2 \cdots \delta_m$. 
By assumption $\mu_k \gamma \prec \mu_l \delta$, appealing again to equation (\ref{twopartitions}) we see that this is only possible if $k = l$ and $\gamma \prec \delta$. 
Consequently $s(\omega) = J_{\mu_k} = s(\xi)$. 
From Lemma \ref{outedges} we get that either $J_\mu$ emits only a single edge, or the $\LL_\J$ labels of the edges out of $J_\mu$ form a partition of $r(\mu)$. 
In the former case we clearly have $\omega_1 = \xi_1$.
In the latter case we denote the labels $\{S_{\alpha_j}\}_j$ and note that $\gamma_1 = \alpha_a$ and $\delta_1 = \alpha_b$ for some $a,b$.  
Since $\alpha_j \not \prec \alpha_i$ for any $i \neq j$ and $\gamma \prec \delta$ we must have that $a=b$. 
Therefore $\gamma_1 = \delta_1$ and both $\omega_1$ and $\xi_1$ are equal to the unique edge out of $J_\mu$ with $\LL_\J$ label $S_{\gamma_1}$. 
We have now established that we always have $\omega_1 = \xi_1$.
Thus the paths $\omega' = \omega_2 \omega_3 \cdots \omega_m$ and $\xi' = \xi_2 \xi_3 \cdots \xi_m$ have the same source, say $J_{\mu_h}$, so 
\begin{align*}
	\LL_s(\omega') \LL_\J(\omega') &= S_{\mu_h} S_{\gamma_2 \gamma_3 \cdots \gamma_m}, \\
	\LL_s(\xi') \LL_\J(\xi') &= S_{\mu_h} S_{\xi_2 \xi_3 \cdots \xi_m},	
\end{align*}
and $\mu_h \gamma_2 \gamma_3 \cdots \gamma_m \prec \mu_h \xi_2 \xi_3 \cdots \xi_m$. 
By the induction hypothesis we conclude that $\omega' = \xi'$. 
Therefore
\[
	\omega = \omega_1 \omega' = \xi_1 \xi' = \xi. \qedhere
\]
\end{proof}

\begin{lemma}
Let $u_\J\in\SS_E$. 
Assume that the coding graph $E_\J$ has only non-negative edges and does not contain any non-positive cycles. 
Assume further that $\Lambda_u$ restricts to an isomorphism of $\D_{L_\Zb(E)}$.
For all $\xi \in \E_\J^*$ we have 
\[
	\LL_s(\xi) \LL_\J(\xi) \LL_r(\xi)^* \in \image(\Lambda_u).
\]
\end{lemma}
\begin{proof}

Let $\omega_1, \omega_2, \ldots, \omega_m$ be an enumeration of the paths in $E_\J$ with $\E$ label $\E(\xi)$. 
From Lemma \ref{imageofpath} we get that 
\[
	\Lambda_u(S_{\E(\xi)}) = \sum_{i=1}^m \LL_s(\omega_i) \LL_\J(\omega_i) \LL_r(\omega_i)^*
\]
Since $E_\J$ only has non-negative edges we can find paths $\gamma_1, \gamma_2, \ldots, \gamma_m$ such that 
\[
	\LL_s(\omega_i) \LL_\J(\omega_i) = S_{\gamma_i}.
\]
Since paths with the same $\E$ label have the same length it follows from Lemma \ref{lem:nosubpathlabel} that for no pair $i,j$, $i \neq j$, is $\gamma_i \prec \gamma_j$.
Therefore 
\[
	P_{\gamma_i} S_{\gamma_j} = \begin{cases} S_{\gamma_j}, & i = j \\ 0, & i \neq j \end{cases} 
\] 
Suppose $\xi = \omega_k$, then 
\begin{align*}
	\LL_s(\xi) \LL_\J(\xi) \LL_r(\xi)^* &= P_{\gamma_k} \LL_s(\omega_k) \LL_\J(\omega_k) \LL_r(\omega_k)^* \\
	&= \sum_{i=1}^m P_{\gamma_k} \LL_s(\omega_i) \LL_\J(\omega_i) \LL_r(\omega_i)^* \\
	&= P_{\gamma_k} \sum_{i=1}^m \LL_s(\omega_i) \LL_\J(\omega_i) \LL_r(\omega_i)^* \\
	&= P_{\gamma_k} \Lambda_u(S_{\E(\xi)}) \in \image(\Lambda_u) \qedhere
\end{align*}
\end{proof}


\section{Action on the path space and transducers}\label{transducers}

Throughout this section, we fix a unitary $u_\J \in \SS_E$ such that the coding graph $E_\J$ only has 
non-negative edges and does not contain any non-positive cycles. 
Furthermore, the coding graph $E_\J$ should be left-synchronizing (see Definition \ref{leftsync}). 
Under these assumptions it follows from Corollary \ref{cor:diagonalauto} that $\lambda_u$ restricts to an automorphism of $\D_{C^*(E)}$.
Since $\D_{C^*(E)} \cong C(E^\infty)$, Gelfand duality implies that $\lambda_u$ induces a self homeomorphism on $E^\infty$, denoted 
$\lambda_u^*$. Here $E^\infty$ denotes the space of one-sided infinite paths on $E$. 
The purpose of this section is to use the coding graph $E_\J$ to describe the action of $(\lambda_u^*)^{-1}$.
(Why the inverse of $\lambda_u^*$? Because Gelfand duality is contravariant).
To ease notation, we put $\psi_u = (\lambda_u^*)^{-1}$.

To get started,  we first give a description of the action $\psi_u$ and recall the definition of a transducer from \cite{bb, russian}.  
The following Lemma \ref{lem:whatpsidoes} follows immediately from the construction of the isomorphisms $\D_{C^*(E)} \cong C(E^\infty)$, so that  
a path $\alpha\in E^\infty$ is identified with the character of $\D_{C^*(E)}$ which maps a projection $P_\mu$ to $1$ or $0$ depending on if 
$\mu$ is an initial word of $\alpha$ or not. 

\begin{lemma} \label{lem:whatpsidoes}
Suppose $\alpha, \beta \in E^\infty$. 
The following are equivalent:
\begin{enumerate}
	\item $\psi_u(\alpha) = \beta$.
	\item For all $n \in \Nb$ there exists $m \in \Nb$ such that 
	\[
		p_{\beta_1 \beta_2 \cdots \beta_n} \lambda_u(p_{\alpha_1 \alpha_2 \cdots \alpha_{m}}) \neq 0.
	\]
	\item For all $n \in \Nb$ there exists $m \in \Nb$ such that
	\[
		s_{\beta_1 \beta_2 \cdots \beta_n}^* \lambda_u(s_{\alpha_1 \alpha_2 \cdots \alpha_{m}}) \neq 0. 
	\]
\end{enumerate}
\end{lemma}
   
\begin{definition}
A transducer is a $5$-tuple $(\A, \B, \SS, s_0, \tau)$ where 
\begin{enumerate}
	\item $\A$, $\B$ are finite alphabets for input and output, respectively, 
	\item $\SS$ is a finite set of states,
	\item $s_0 \in \SS$ is the initial state, and 
	\item $\tau \colon \SS \times \A \to \SS \times \B^*$ is the transition function.  
\end{enumerate} 
\end{definition}   
   
Note that we allow the input and outupt alphabets to differ, in contrast to \cite{bb}.
Given a transducer $T = (\A, \B, \SS, s_0, \tau)$ and an infinite word $\alpha \in \A^\infty$ we recursively define a state sequence $(s_i)_{i=0}^\infty$ and an output sequence $(\beta_i)_{i=1}^\infty$ by 
\[
	(s_i, \beta_i) = \tau(s_{i-1}, \alpha_i).
\]
The output word of $T$ on $\alpha$ is the concatenation $\beta = \beta_1 \beta_2 \cdots$. 
This will be  written as $\alpha[T]\beta$.

Now, we turn to construction of a transducer which induces action $\psi_u$ on the infinite path space $E^\infty$. We do so by composing two 
transducers with distinct inputs--outputs, related to the two graphs $E$ and $E_\J$. At the heart of this construction lies the observation that under the hypothesis of Theorem \ref{diagonalauto} the infinite path spaces $E^\infty$ and $E_\J^\infty$ may be canonically identified. For the sake of 
greater clarity, we break the proof Theorem \ref{ducer} into several natural steps, given below as lemmas \ref{lem:syncsamefirst}--\ref{psiasEJ}. 

\begin{lemma} \label{lem:syncsamefirst}
Suppose $E_\J$ is left-synchronizing with delay $m$ and let $\gamma \in E^{m+2}$ be given.
If $\omega, \xi \in E_\J^{m+1}$ and 
\[
	\E(\omega) = \gamma_1 \gamma_2 \cdots \gamma_{m+2} = \E(\gamma),
\]
then $\omega_1 = \gamma_1$. 
\end{lemma}
\begin{proof}
Define paths 
\begin{align*}
	\omega' &= \omega_1 \omega_2 \cdots \omega_m, \quad
	\omega'' = \omega_2 \omega_3 \cdots \omega_{m+1}, \\
	\xi' &= \xi_1 \xi_2 \cdots \xi_m, \quad
	\xi'' = \xi_2 \xi_3 \cdots \xi_{m+1}. 
\end{align*}
Note that $\E(\omega') = \E(\xi')$ and $\E(\omega'') = \E(\xi'')$.
From synchronization we get that $s(\omega) = s(\omega') = s(\xi') = s(\xi)$ and $s(\omega'') = s(\xi'')$.
Since $E_\J$ has no multiple edges we must have that both $\omega_1$ and $\xi_1$ is the unique edge in $E_\J$ from $s(\omega)$ to $s(\omega'')$. 
\end{proof}

\begin{lemma} \label{transducetoEJ}
Suppose that $E_\J$ is left-synchronizing with delay $m$, and let $\alpha \in E^\infty$. Then 
there exists a unique $\omega \in E_\J^\infty$ such that $\E(\omega) = \alpha$. 
Moreover, there exists a transducer $T$ with input alphabet $E^1$ and output alphabet $E_\J^1$ such that $\alpha[T]\omega$.
\end{lemma}
\begin{proof}
By injectivity of $\lambda_u$ and Lemma \ref{imageofpath}, there exists an $\omega$ with $\E(\omega) = \alpha$. 
By Lemma \ref{lem:syncsamefirst}, such an $\omega$ is unique.

From Lemma \ref{lem:syncsamefirst} we see that there exists a map $\phi \colon E^{m+2} \to E_\J^1$ such that for any $\beta \in E^{m+2}$ and any $\xi \in E_J^{m+1}$ with $\E(\xi) = \beta$ we have $\xi_1 = \phi(\beta)$.
We can extend $\phi$ to $(E^1)^{m+2}$ by setting $\phi(\beta) = \emptyset$ for any word $\beta \in (E^1)^{m+2}$ which is not a path. 
Define a sliding block code $\Phi \colon E^\infty \to E_{\J}^\infty$ by 
\[
	\Phi(\gamma) = \phi(\gamma_1 \gamma_2 \cdots \gamma_{m+2}) \phi(\gamma_2 \gamma_3 \cdots \gamma_{m+3}) \cdots.
\]	
By construction, we have $\Phi(\alpha) = \omega$.
To complete the proof we observe that every sliding block code can be implemented by a transducer, see for instance \cite[Figure 2.3]{Rigo} and the surrounding discussion therein.
\end{proof}

\begin{lemma} \label{transducefromEJ}
There exists a transducer $T$ with input alphabet $E_\J^1$ and output alphabet $E^1$ such that 
if we input $\omega \in E_\J^\infty$ then the output is 
\[
	\mu \gamma_1 \gamma_2 \cdots,
\]
where $s_\mu = \LL_s(\omega)$ and $s_{\gamma_i} = \LL_\J(\omega_i)$, $i=1,2,\ldots$.
\end{lemma}
\begin{proof}
The state space of our transducer will be $E_\J^1$ and a special symbol $s_0$, which is the initial state.
The transition function $\tau$ is defined as follows:
From the initial state we move to the state of the first letter we read, and we write $\LL_s$ label. 
In symbols:
\[
	\tau(s_0, e) = (e, \mu), \text{ where } s_\mu = \LL_s(s(e)).
\] 
From any non-initial state, we still move to the state of the letter we read, but we now write the $\LL_\J$ label of the state we are leaving. 
In symbols
\[
	\tau(e, f) = (f, \gamma), \text{ where } s_\gamma = \LL_\J(e). 
\]
$T$ has the desired property by construction. 
\end{proof}

\begin{lemma} \label{psiasEJ}
If $\alpha \in E^\infty$ and $\omega \in E_\J^\infty$ is the unique path with $\E(\omega) = \alpha$ (as in Lemma \ref{transducetoEJ}), then 
\[
	\psi_u(\alpha) = \mu \gamma_1 \gamma_2 \cdots,
\]
where $s_\mu = \LL_s(\omega)$ and $s_{\gamma_i} = \LL_\J(\omega_i)$, $i=1,2,\ldots$.
\end{lemma}
\begin{proof}
Suppose that $E_\J$ is left-synchronizing with delay $k$. 
To ease notation, we put $\beta = \mu \gamma_1 \gamma_2 \cdots$ and write $\delta_{[m]} = \delta_1 \delta_2 \cdots \delta_m$ for any path $\delta$ 
(possibly infinite) which begins with $\delta_1 \delta_2 \cdots \delta_m$.
By Lemma \ref{lem:whatpsidoes}, it suffices to show that for all $n \in \Nb$ there exists an $m \in \Nb$ such that 
\[
	\lambda_u(s_{\alpha_{[m]}}) = s_{\beta_{[n]}} x. 
\]	
This is equivalent to showing that for all $n \in \Nb$ there exists an $m \in \Nb$ such that
\[
	\lambda_u(s_{\alpha_{[m]}}) = s_{\mu \gamma_1 \gamma_2 \cdots \gamma_n} x. 
\]	
It is this last claim we will establish. 

Let $n \in \Nb$ be given. 
From Lemma \ref{imageofpath} we get for all $m \in \Nb$ that 
\[
	\lambda_u(s_{\alpha_{[m]}}) = \sum_{\E(\xi) = \alpha_{[m]}} \LL_s(\xi) \LL_\J(\xi) \LL_r(\xi)^*.
\]
Let $m = k + 1 + n$. 
Then it follows from Lemma \ref{lem:syncsamefirst} that if $\xi$ is a path with $\E(\xi) = \alpha_{[m]}$ then $\xi_{[n]} = \omega_{[n]}$.
Consequently   
\begin{align*}
	\lambda_u(s_{\alpha_{[m]}}) &= \sum_{\E(\xi) = \alpha_{[m]}} \LL_s(\xi) \LL_\J(\xi) \LL_r(\xi)^* \\
	&= \sum_{\E(\xi) = \alpha_{[m]}} \LL_s(\omega) \LL_\J(\omega_{[n]}) \LL_\J(\xi_{n+1} \xi_{n+2} \cdots \xi_m) \LL_r(\xi)^* \\
	&= \LL_s(\omega) \LL_\J(\omega_{[n]}) \sum_{\E(\xi) = \alpha_{[m]}} \LL_\J(\xi_{n+1} \xi_{n+2} \cdots \xi_m) \LL_r(\xi)^* \\
	&= s_\mu s_{\gamma_1 \gamma_2 \cdots \gamma_n}  x.  \qedhere
\end{align*}
\end{proof}

\begin{theorem}\label{ducer}
There exists a transducer $T$ with input and output alphabet $E^1$ such that for all $\alpha \in E^\infty$ we have
\[
	\alpha[T]\psi_u(\alpha).
\]
\end{theorem}
\begin{proof}
It is described in \cite[Section 3.2]{russian}  how the composition of two transducers is again a transducer, 
and that if $a[S]b$ and $b[S']c$ then $a[S' \circ S]c$.  
Let $T$ be the composition of the transducers from lemmas \ref{transducetoEJ} and \ref{transducefromEJ}. 
Then it follows from Lemma \ref{psiasEJ} that $\alpha[T]\psi_u(\alpha)$.  
\end{proof}

\begin{remark} \label{rmk:howtofindpsi} \rm
If we have a unitary $u_\J$ with left-synchronizing (delay $m$) coding graph $u_\J$ and we want to describe the action of $\psi_u$, then the proof of Lemma \ref{psiasEJ} gives us a recipe.
First, as in Lemma \ref{transducetoEJ}, find a map $\phi \colon E^{m+2} \to E_\J^1$ with the property that for all $\omega \in E_\J^{m+1}$ we have $\phi(\E(\omega)) = \omega_1$.
Then define two maps $S,L \colon E_\J^1 \to E^*$ by 
\begin{align*}
	S(\omega) &= \mu, \text{ where } \LL_r(\omega) = s_\mu, \\
	L(\omega) &= \gamma, \text{ where } \LL_\J(\omega) = s_\gamma.
\end{align*}
Then for any $\alpha \in E^\infty$ we have
\[
	\psi_u(\alpha) = S(\phi(\alpha_1 \alpha_2 \cdots \alpha_{m+2})) L(\phi(\alpha_1 \alpha_2 \cdots \alpha_{m+2})) L(\phi(\alpha_2 \alpha_3 \cdots \alpha_{m+3})) \cdots.
\]
Alternatively, we can define $K \colon E^\infty \to E^\infty$ by 
\[	
	K(\alpha) = L(\phi(\alpha_1 \alpha_2 \cdots \alpha_{m+2})) K(\alpha_2 \alpha_3 \cdots),
\]
and then 
\[
	\psi_u(\alpha) = S(\phi(\alpha_1 \alpha_2 \cdots \alpha_{m+2})) K(\alpha)
\]
\end{remark}


\section{Examples}

In examples \ref{exam1} and \ref{exam2} below, we work with the Leavitt path algebra $L_{2, \Zb}$ corresponding to the Cuntz algebra $\OO_2$, i.e. the Leavitt path algebra of a graph consisting of one vertex, denoted $\emptyset$, and two edges, denoted $1$ and $2$, respectively. 

\begin{example}\label{exam1}
\rm
Consider a unitary 
$$ u = S_{11}S_1^* + S_{12}S_{21}^* + S_2S_{22}^* $$ 
in $L_{2, \Zb}$ and the corresponding inner automorphism $\ad(u) = \Lambda_{u\Phi(u^*)}$. We have 
$$ u\Phi(u^*) = P_{111} + S_{12}S_{211}^* + S_{1121}S_{112}^* + S_{21}S_{212}^* + S_{1122}S_{12}^* + P_{22}. $$
The coding graph $(E_\J,\LL_\J)$ of this representation of the unitary $u\Phi(u^*)$ is shown in the figure below. 
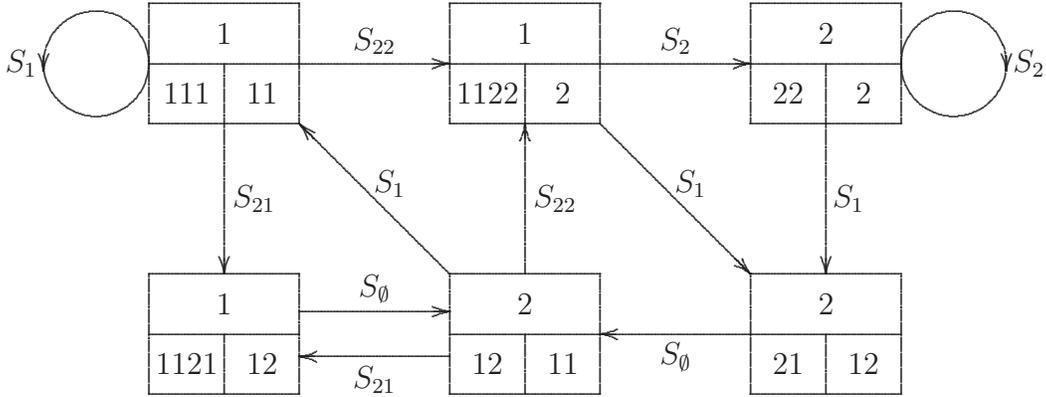
\[ \beginpicture
\setcoordinatesystem units <1cm,1cm>
\setplotarea x from -7 to 7, y from -3 to 3
\setlinear 
\plot -1 2.6  1 2.6  1 1  -1 1  -1 2.6 /
\plot -1 1.8  1 1.8 /
\plot 0 1  0 1.8 /
\put {$1$} at 0 2.2
\put {$1122$} at -0.5 1.4
\put {$2$} at 0.5 1.4
\setlinear 
\plot 3 2.6  5 2.6  5 1  3 1  3 2.6 /
\plot 3 1.8  5 1.8 /
\plot 4 1  4 1.8 /
\put {$2$} at 4 2.2
\put {$22$} at 3.5 1.4
\put {$2$} at 4.5 1.4
\setlinear 
\plot -5 2.6  -3 2.6  -3 1  -5 1  -5 2.6 /
\plot -5 1.8  -3 1.8 /
\plot -4 1  -4 1.8 /
\put {$1$} at -4 2.2
\put {$111$} at -4.5 1.4
\put {$11$} at -3.5 1.4
\setlinear 
\plot -5 -2.6  -3 -2.6  -3 -1  -5 -1  -5 -2.6 /
\plot -5 -1.8  -3 -1.8 /
\plot -4 -2.6  -4 -1.8 /
\put {$1$} at -4 -1.4
\put {$1121$} at -4.5 -2.2
\put {$12$} at -3.5 -2.2
\setlinear 
\plot -3 1.8  -1 1.8 /
\plot 1 1.8  3 1.8 /
\plot -4 1  -4 -1 /
\plot 4 1  4 -1 /
\plot 0 1  0 -1 /
\plot -3 1  -1 -1 /
\plot 1 1  3 -1 /
\plot 1 -1.8  3 -1.8 /
\plot -3 -1.5  -1 -1.5 /
\plot -3 -2.1  -1 -2.1 /
\plot 5 -2.6  3 -2.6  3 -1  5 -1  5 -2.6 /
\plot 5 -1.8  3 -1.8 /
\plot 4 -2.6  4 -1.8 /
\put {$2$} at 4 -1.4
\put {$12$} at 4.5 -2.2
\put {$21$} at 3.5 -2.2
\setlinear 
\plot  1 -1  1 -2.6  -1 -2.6  -1 -1  1 -1  /
\plot -1 -1.8  1 -1.8 /
\plot 0 -2.6  0 -1.8 /
\put {$2$} at 0 -1.4
\put {$11$} at 0.5 -2.2
\put {$12$} at -0.5 -2.2
\circulararc  360 degrees from 5 1.8  center at 5.7 1.8
\circulararc  360 degrees from -5 1.8  center at -5.7 1.8
\arrow <0.235cm> [0.2,0.6] from -6.4 1.9 to -6.4 1.7
\arrow <0.235cm> [0.2,0.6] from 6.4 1.9 to 6.4 1.7
\arrow <0.235cm> [0.2,0.6] from -1.2 1.8 to -1 1.8
\arrow <0.235cm> [0.2,0.6] from 2.8 1.8 to 3 1.8
\arrow <0.235cm> [0.2,0.6] from -4 -0.8 to -4 -1
\arrow <0.235cm> [0.2,0.6] from 4 -0.8 to 4 -1
\arrow <0.235cm> [0.2,0.6] from 0 0.8 to 0 1
\arrow <0.235cm> [0.2,0.6] from 1.2 -1.8 to 1 -1.8
\arrow <0.235cm> [0.2,0.6] from -1.2 -1.5 to -1 -1.5
\arrow <0.235cm> [0.2,0.6] from -2.8 -2.1 to -3 -2.1
\arrow <0.235cm> [0.2,0.6] from -2.8 0.8 to -3 1
\arrow <0.235cm> [0.2,0.6] from 2.8 -0.8 to 3 -1
\put {$S_1$} at -6.7 1.8
\put {$S_2$} at 6.7 1.8
\put {$S_{22}$} at -2 2.1
\put {$S_2$} at 2 2.1
\put {$S_{21}$} at -3.6 0
\put {$S_{22}$} at 0.4 0
\put {$S_{1}$} at 4.3 0
\put {$S_1$} at -1.8 0.2
\put {$S_1$} at 2.2 0.2
\put {$S_\emptyset$} at 2 -2.1
\put {$S_\emptyset$} at -2 -1.2
\put {$S_{21}$} at -2 -2.4
\endpicture \]
One can easily check that the labeled graph $(E_\J,\E)$ is left-synchronizing. Thus $\Lambda_u$ restricts to an automorphism of the diagonal MASA 
$\D_{L_\Zb(E)}$, by Theorem \ref{diagonalauto}. 
\end{example}

\begin{example}\label{exam2}
\rm
The unitary 
$$ u = P_{122} + S_{11}S_{121}^* + S_{121}S_{11}^* + P_2 $$ 
gives rise to an outer automorphism of order two of $C^*(E)$, \cite[Formula (10)]{CHSam}. 
The corresponding coding graph is shown below. 
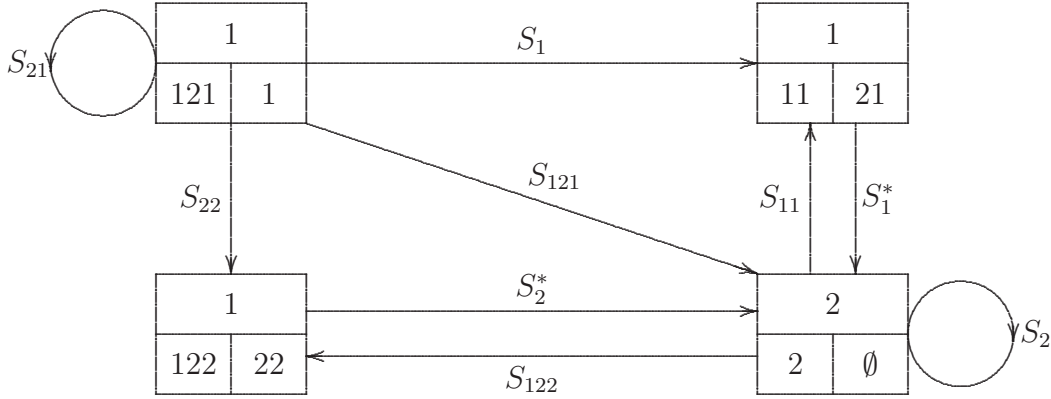
\[ \beginpicture
\setcoordinatesystem units <1cm,1cm>
\setplotarea x from -7 to 7, y from -3 to 3
\put {$\emptyset$} at 4.5 -2.2
\setlinear 
\plot 3 2.6  5 2.6  5 1  3 1  3 2.6 /
\plot 3 1.8  5 1.8 /
\plot 4 1  4 1.8 /
\put {$1$} at 4 2.2
\put {$11$} at 3.5 1.4
\put {$21$} at 4.5 1.4
\setlinear 
\plot -5 2.6  -3 2.6  -3 1  -5 1  -5 2.6 /
\plot -5 1.8  -3 1.8 /
\plot -4 1  -4 1.8 /
\put {$1$} at -4 2.2
\put {$121$} at -4.5 1.4
\put {$1$} at -3.5 1.4
\setlinear 
\plot -5 -2.6  -3 -2.6  -3 -1  -5 -1  -5 -2.6 /
\plot -5 -1.8  -3 -1.8 /
\plot -4 -2.6  -4 -1.8 /
\put {$1$} at -4 -1.4
\put {$122$} at -4.5 -2.2
\put {$22$} at -3.5 -2.2
\put {$2$} at 4 -1.4
\put {$2$} at 3.5 -2.2
\setlinear 
\plot 3.7 1  3.7 -1 /
\plot 4.3 1  4.3 -1 /
\plot -3 1.8  3 1.8 /
\plot -3 1  3 -1 /
\plot -4 1  -4 -1 /
\plot -3 -2.1  3 -2.1 /
\plot -3 -1.5  3 -1.5 /
\plot 3 -2.6  5 -2.6  5 -1  3 -1  3 -2.6 /
\plot 3 -1.8  5 -1.8 /
\plot 4 -2.6  4 -1.8 /
\circulararc  360 degrees from 5 -1.8  center at 5.7 -1.8
\circulararc  360 degrees from -5 1.8  center at -5.7 1.8
\arrow <0.235cm> [0.2,0.6] from -6.4 1.9 to -6.4 1.7
\arrow <0.235cm> [0.2,0.6] from 6.4 -1.7 to 6.4 -1.9
\arrow <0.235cm> [0.2,0.6] from 2.8 1.8 to 3 1.8
\arrow <0.235cm> [0.2,0.6] from -4 -0.8 to -4 -1
\arrow <0.235cm> [0.2,0.6] from 3.7 0.8 to 3.7 1
\arrow <0.235cm> [0.2,0.6] from 4.3 -0.8 to 4.3 -1
\arrow <0.235cm> [0.2,0.6] from 2.8 -1.5 to 3 -1.5
\arrow <0.235cm> [0.2,0.6] from -2.8 -2.1 to -3 -2.1
\arrow <0.235cm> [0.2,0.6] from 2.7 -0.9   to 3 -1
\put {$S_{21}$} at -6.7 1.8
\put {$S_2$} at 6.7 -1.8
\put {$S_1$} at 0 2.1
\put {$S_{22}$} at -4.4 0
\put {$S_1^*$} at 4.6 0
\put {$S_{11}$} at 3.3 0
\put {$S_2^*$} at 0 -1.2
\put {$S_{122}$} at 0 -2.4
\put {$S_{121}$} at 0.3 0.3
\endpicture \]
Since this coding graph contains negative edges incoming to the vertex $(2,2)$, we perform splitting at this vertex. 
The resulting coding graph $(E_\J,\LL_\J)$ of this new representation of the unitary $u$, shown below, contains only non-negative edges. 
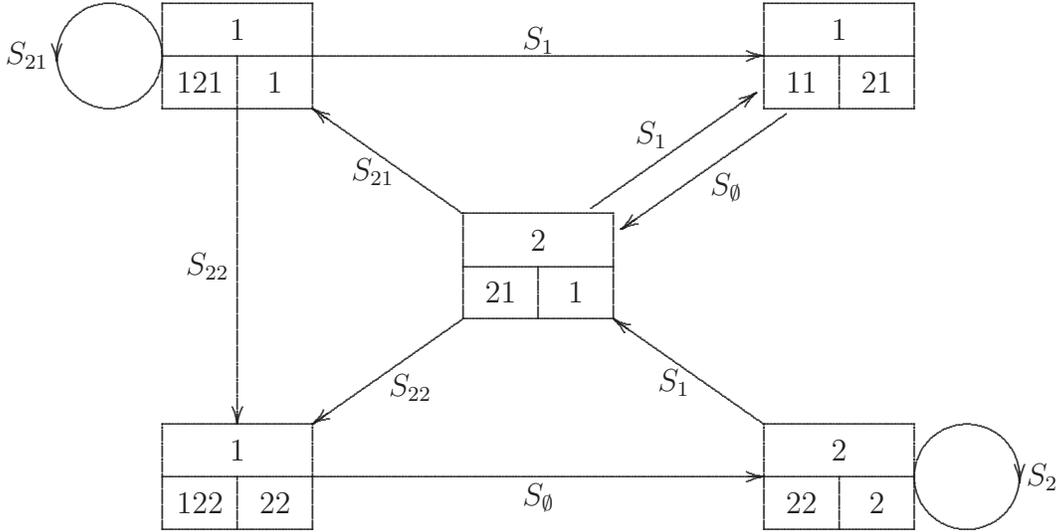
\[ \beginpicture
\setcoordinatesystem units <1cm,0.7cm>
\setplotarea x from -7 to 7, y from -5 to 5
\setlinear 
\plot -5 5  -3 5  -3 3  -5 3  -5 5 /  
\plot -5 4  -3 4 /
\plot -4 3  -4 4 /
\plot 3 5  5 5  5 3  3 3  3 5 /
\plot 3 4  5 4 /
\plot 4 3  4 4 /
\plot -5 -5  -3 -5  -3 -3  -5 -3  -5 -5 /  
\plot -5 -4  -3 -4 /
\plot 4 -5  4 -4 /
\plot 3 -5  5 -5  5 -3  3 -3  3 -5 /
\plot 3 -4  5 -4 /
\plot -4 -5  -4 -4 /
\plot -1 1  1 1  1 -1  -1 -1  -1 1 /
\plot -1 0  1 0 /
\plot 0 -1 0 0 /
\circulararc  360 degrees from -5 4  center at -5.7 4
\circulararc  360 degrees from 5 -4  center at 5.7 -4
\setlinear
\plot -3 4  3 4 /
\plot -3 -4  3 -4 /
\plot -4 3  -4 -3 /
\plot -3 3  -1 1 /
\plot -3 -3  -1 -1 /
\plot 1 -1  3 -3 /
\plot 1.1 0.7  3.3 2.9 /
\plot 0.7 1.1  2.9 3.3 /
\arrow <0.235cm> [0.2,0.6] from 2.8 4 to 3 4
\arrow <0.235cm> [0.2,0.6] from -6.4 4.1 to -6.4 3.9
\arrow <0.235cm> [0.2,0.6] from 6.4 -3.9 to 6.4 -4.1
\arrow <0.235cm> [0.2,0.6] from 2.7 3.1  to 2.9 3.3
\arrow <0.235cm> [0.2,0.6] from 1.3 0.9   to 1.1 0.7
\arrow <0.235cm> [0.2,0.6] from -4 -2.8 to -4 -3
\arrow <0.235cm> [0.2,0.6] from 2.8 -4 to 3 -4
\arrow <0.235cm> [0.2,0.6] from -2.8 2.8   to -3 3
\arrow <0.235cm> [0.2,0.6] from 1.2 -1.2 to 1 -1
\arrow <0.235cm> [0.2,0.6] from -2.8 -2.8   to -3 -3
\put {$S_1$} at 0 4.3
\put {$S_{21}$} at -6.8 4
\put {$S_2$} at 6.7 -4
\put {$S_{22}$} at -4.4 0
\put {$S_\emptyset$} at 0 -4.4
\put {$S_{21}$} at -2.2 1.8
\put {$S_{22}$} at -1.7 -2.3
\put {$S_{1}$} at 1.8 -2.2
\put {$S_1$} at 1.5 2.5
\put {$S_\emptyset$} at 2.5 1.5
\put {$1$} at -4 4.5
\put {$121$} at -4.5 3.5
\put {$1$} at -3.5 3.5
\put {$1$} at 4 4.5
\put {$11$} at 3.5 3.5
\put {$21$} at 4.5 3.5
\put {$2$} at 0 0.5
\put {$21$} at -0.5 -0.5
\put {$1$} at 0.5 -0.5
\put {$1$} at -4 -3.5
\put {$122$} at -4.5 -4.5
\put {$22$} at -3.5 -4.5
\put {$2$} at 4 -3.5
\put {$22$} at 3.5 -4.5
\put {$2$} at 4.5 -4.5
\endpicture \]
The labeled graph $(E_\J,\E)$ is left-synchronizing and thus $\Lambda_u$ restricts to an automorphism of the diagonal MASA $\D_{L_\Zb(E)}$, 
by Theorem \ref{diagonalauto}. 

We now follow Remark \ref{rmk:howtofindpsi} to describe $\psi_u$. 
First we note that $(E_\J,\E)$ is left-synchronizing with delay $2$, and hence an $\E$-label of length $4$ will uniquely determine the first edge in any path.  
There are $2^4 = 16$ length $4$ paths in $E$. 
The table below shows what each of these gets mapped to under $\phi$, $S \circ \phi$ and $L \circ \phi$. 
Recall that $[J_{121}, J_{11}]$ denotes the horizontal edge at the top of the graph.

\begin{center}
\begin{tabular}{M|MMM}
E^4 & \phi & S \circ \phi & L \circ \phi \\
\hline
 1111 & [J_{121}, J_{121}]	& 121	& 21 \\
 1112 & [J_{121}, J_{121}] 	& 121	& 21 \\
 1121 & [J_{121}, J_{11}]	& 121	& 1 \\
 1122 & [J_{121}, J_{122}]	& 121	& 22 \\
 1211 & [J_{11}, J_{21}]	& 11	& \emptyset \\
 1212 & [J_{11}, J_{21}] 	& 11 	& \emptyset \\
 1221 & [J_{122}, J_{21}] 	& 122 	& \emptyset \\
 1222 & [J_{122}, J_{21}]	& 122 	& \emptyset \\
 2111 & [J_{21}, J_{121}]	& 21	& 21 \\
 2112 & [J_{21}, J_{121}] 	& 21	& 21 \\
 2121 & [J_{21}, J_{11}]	& 21	& 1 \\
 2122 & [J_{21}, J_{122}]	& 21	& 22 \\
 2211 & [J_{22}, J_{21}]	& 22	& 1 \\
 2212 & [J_{22}, J_{21}] 	& 22 	& 1 \\
 2221 & [J_{22}, J_{22}] 	& 22 	& 2 \\
 2222 & [J_{22}, J_{22}]	& 22 	& 2 
\end{tabular}
\end{center}

So if we for instance consider the path $\alpha = (112)^\infty = 112112112112112 \cdots$, then there are only three possible length four subpaths, namely $1121$, $1211$, and $2112$.
So the relevant information in the above table reduces to.

\begin{center}
\begin{tabular}{M|MM}
E^4 & S \circ \phi & L \circ \phi \\
\hline
 1121 & 121	& 1 \\
 1211 & 11	& \emptyset \\
 2112 & 21	& 21 \\
\end{tabular}
\end{center} 
We can now compute 
\begin{align*}
	\psi_u((112)^\infty) &= (S \circ \phi)(1121) \left( ((L \circ \phi)(1121)) ((L \circ \phi)(1211)) ((L \circ \phi)(2112)) \right)^\infty \\
	&= 121 ((1) (\emptyset) (21))^\infty = 121 (121)^\infty = (121)^\infty.
\end{align*}
Similarly 
\begin{align*}
	\psi_u((121)^\infty) &= (S \circ \phi)(1211) \left( ((L \circ \phi)(1211)) ((L \circ \phi)(2112)) ((L \circ \phi)(1121)) \right)^\infty \\
	&= 11 ((\emptyset) (21) (1))^\infty = 11 (211)^\infty = (112)^\infty.
\end{align*}
As it should be, since $\psi_u$ has order $2$.

\end{example}

\begin{example}\label{exam3}
\rm
In this example, we work with another Leavitt path algebra, given by the following graph $F$. 
It should be noted that the corresponding graph Leavitt path algebra is again isomorphic to $L_{2, \Zb}$. 
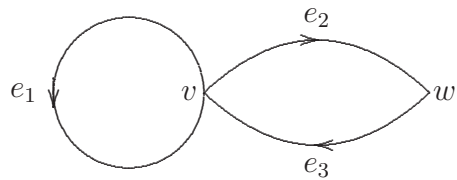
\[ \beginpicture
\setcoordinatesystem units <2cm,2cm>
\setplotarea x from -2 to 2, y from -0.7 to 0.7
\circulararc  360 degrees from 0 0  center at -0.5 0
\setquadratic
\plot 0 0  0.75 0.35  1.5 0 /
\plot 0 0  0.75 -0.35  1.5 0 /
\arrow <0.235cm> [0.2,0.6] from -1 0.1 to -1 -0.1
\arrow <0.235cm> [0.2,0.6] from  0.74 0.35 to 0.76 0.35
\arrow <0.235cm> [0.2,0.6] from  0.76 -0.35  to 0.74 -0.35
\put {$e_1$} at -1.2 0
\put {$e_2$} at 0.75 0.5
\put {$e_3$} at 0.75 -0.5
\put {$v$} at -0.1 0
\put {$w$} at 1.6 0
\endpicture \]
The generating partial isometries corresponding to the three edges will be denoted $S_1$, $S_2$ and $S_3$, respectively. We consider a unitary 
$$ u = S_{11}S_1^* + S_{12}S_2^* + S_2S_{31}^* + S_3S_{32}^*, $$
whose coding graph $(F_\J,\LL_\J)$ is shown below. 
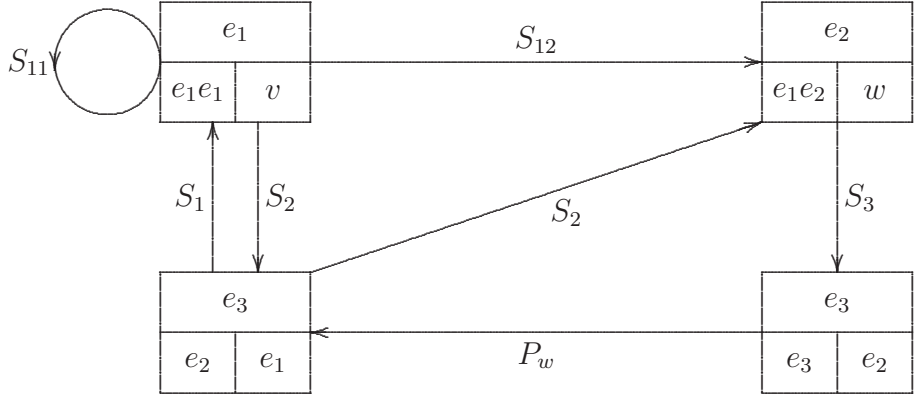
\[ \beginpicture
\setcoordinatesystem units <1cm,1cm>
\setplotarea x from -7 to 7, y from -3 to 3
\put {$e_2$} at 4.5 -2.2
\setlinear 
\plot 3 2.6  5 2.6  5 1  3 1  3 2.6 /
\plot 3 1.8  5 1.8 /
\plot 4 1  4 1.8 /
\put {$e_2$} at 4 2.2
\put {$e_1e_2$} at 3.5 1.4
\put {$w$} at 4.5 1.4
\setlinear 
\plot -5 2.6  -3 2.6  -3 1  -5 1  -5 2.6 /
\plot -5 1.8  -3 1.8 /
\plot -4 1  -4 1.8 /
\put {$e_1$} at -4 2.2
\put {$e_1e_1$} at -4.5 1.4
\put {$v$} at -3.5 1.4
\setlinear 
\plot -5 -2.6  -3 -2.6  -3 -1  -5 -1  -5 -2.6 /
\plot -5 -1.8  -3 -1.8 /
\plot -4 -2.6  -4 -1.8 /
\put {$e_3$} at -4 -1.4
\put {$e_2$} at -4.5 -2.2
\put {$e_1$} at -3.5 -2.2
\put {$e_3$} at 4 -1.4
\put {$e_3$} at 3.5 -2.2
\setlinear 
\plot 4 1  4 -1 /
\plot -3 1.8  3 1.8 /
\plot -3 -1  3 1 /
\plot 3 -2.6  5 -2.6  5 -1  3 -1  3 -2.6 /
\plot 3 -1.8  5 -1.8 /
\plot 4 -2.6  4 -1.8 /
\plot -3 -1.8  3 -1.8 /
\plot -4.3 1  -4.3 -1 /
\plot -3.7 1  -3.7 -1 /
\circulararc  360 degrees from -5 1.8  center at -5.7 1.8
\arrow <0.235cm> [0.2,0.6] from -6.4 1.9 to -6.4 1.7
\arrow <0.235cm> [0.2,0.6] from 2.8 1.8 to 3 1.8
\arrow <0.235cm> [0.2,0.6] from -2.8 -1.8 to -3 -1.8
\arrow <0.235cm> [0.2,0.6] from 4 -0.8 to 4 -1
\arrow <0.235cm> [0.2,0.6] from -4.3 0.8 to -4.3 1
\arrow <0.235cm> [0.2,0.6] from -3.7 -0.8 to -3.7 -1
\arrow <0.235cm> [0.2,0.6] from 2.7 0.9   to 3 1
\put {$S_{11}$} at -6.75 1.8
\put {$S_{12}$} at 0 2.1
\put {$S_{1}$} at -4.6 0
\put {$S_2$} at -3.4 0
\put {$S_3$} at 4.3 0
\put {$S_{2}$} at 0.4  -0.2
\put {$P_w$} at 0 -2.1
\endpicture \]
The labeled graph $(F_\J,\E)$ is left-synchronizing, and thus $\Lambda_\J\in\aut(\D_{L_\Zb(F)})$. 
However, $\Lambda_\J\not\in\aut(L_\Zb(F))$. 
Indeed, we have $\Lambda_\J(S_1)=S_{11}$, $\Lambda_\J(S_2)=S_{12}$ and $\Lambda_\J(S_3)=S_2S_1^*+S_3S_2^*$. 
Thus $\Lambda_\J$ maps all three generators to elements that are homogeneous of degree $2$.
Hence the image of $\Lambda_\J$ will be in the span of homogeneous elements of even degree, which is a proper subalgebra of $L_\Zb(F)$.
The same conclusion holds at the level of the graph $C^*$-algebra. 
\end{example}


\medskip\noindent
Rune Johansen\\
Department of Mathematics\\
University of Copenhagen\\
Universitetsparken 5, 2100 K\o benhavn \O, Denmark\\
E-mail: rune@cyx.dk\\

\smallskip\noindent
Adam P. W. S{\o}rensen\\
Department of Mathematics\\
University of Oslo\\
PO BOX 1053 Blindern, N-0316 Oslo, Norway\\
E-mail: apws@math.uio.no\\

\smallskip\noindent
Wojciech Szyma{\'n}ski\\
Department of Mathematics and Computer Science \\
The University of Southern Denmark \\
Campusvej 55, DK--5230 Odense M, Denmark \\
E-mail: szymanski@imada.sdu.dk

\end{document}